\documentclass{amsart}
\usepackage{amsmath,amssymb,amsthm, tikz-cd, hyperref}
\usepackage{xcolor}
\usepackage{soul}

\newtheorem{theorem}{Theorem}[section]
\newtheorem{lemma}[theorem]{Lemma}

\newtheorem{proposition}[theorem]{Proposition}
\newtheorem{question}{Question}
\theoremstyle{definition}

\newtheorem{example}[theorem]{Example}

\newtheorem{remark}[theorem]{Remark}
\newtheorem*{overview}{Overview}
\newtheorem*{acknowledgements}{Acknowledgements}

\newtheorem{thmintro}{Theorem}

\newcommand{\co}{\colon\thinspace}

\newcommand{\del}{\partial}

\DeclareMathOperator{\diam}{diam}
\DeclareMathOperator{\UW}{UW}
\DeclareMathOperator{\length}{length}

\DeclareMathOperator{\RP2}{\mathbb{RP}^2}

\begin{document}

\title{1-Uryson width and covers}
\author{Hannah Alpert}
\address{Auburn University, 221 Parker Hall, Auburn, AL 36849, USA}
\email{hca0013@auburn.edu}
\author{Arka Banerjee}
\address{Ramakrishna Mission Vivekananda Educational and Research Institute, Belur, West Bengal 711202, India}
\email{banerjee20arka@gmail.com}
\author{Panos Papasoglu}
\address{Mathematical Institute, University of Oxford, Andrew Wiles Building, Woodstock Road, Oxford OX2 6GG, U.K.}
\email{Panagiotis.Papazoglou@maths.ox.ac.uk}

\begin{abstract}
We investigate the following question: Do there exist Riemannian polyhedra $X$ such that the 1-Uryson width of their universal covers $\UW_1(\widetilde{X})$ is bounded but $\UW_1(X)$ is arbitrarily large?
We rule out two specific cases: when $\pi_1(X)$ is virtually cyclic and when $X$ is a Riemannian surface.
More specifically,
we show that if $X$ is a compact polyhedron with a virtually cyclic fundamental group, then its 1-Uryson width is bounded by the 1-Uryson width  of its universal cover $\widetilde{X}$, up to a multiplicative constant. Precisely:
$$\UW_1(X) \leq 6 \cdot \UW_1(\widetilde{X}).$$
We show that if $X$ is a compact Riemannian surface then 
$$\UW_1(X) \leq \UW_1(\widetilde{X}).$$
Furthermore, we show that if there exist spaces $X$ for which $\UW_1(\widetilde{X})$ is bounded while $\UW_1(X)$ is arbitrarily large, then such examples must already appear in low dimensions. In particular, such $X$ can be found among Riemannian $2$-complexes.
\end{abstract}
\maketitle

\section{Introduction}
Uryson width originates in classical dimension theory. However, recently it
has proved to be a useful notion in Riemannian Geometry. The main applications
are in systolic inequalities (\cite{Gromov83}, \cite{Guth11}) and in the study of manifolds of positive scalar curvature \cite{Gromov20}. We recall the definition: the $k$-dimensional Uryson width of a metric space $X$, denoted $\UW_k(X)$, is
the infimal $\varepsilon$ such that there exists a continuous map $f : X \to  Y$ , where $Y$ is a
$k$-dimensional simplicial complex, for which each fiber $f^{-1}(y)$ has diameter at most $\varepsilon$. 

In this paper, we are interested in characterizing spaces with small 1-Uryson width.
More precisely, we have the following (posed originally in \cite{ABG21} 
for 3-manifolds).

\begin{question}[Main question]\label{main}
Let $\{X_n\}_{n=1}^\infty$ be a sequence of compact geodesic metric spaces and suppose that their universal covers satisfy $\UW_1(\widetilde{X_n}) \leq 1$.  Must the sequence $\UW_1(X_n)$ be bounded?
\end{question}

While one might intuitively expect the answer to be yes, very little is known about this question. 
M. Katz~\cite {Ka88} gave an affirmative answer in the case $\pi_1(X_n)$ is finite or $\mathbb Z$. Balitskiy-Berdnikov \cite{BB21} give an estimate of $\UW_1(X)$ in terms of $\dim H_1(X; \mathbb{Z}_2)$ assuming instead that the unit balls in $X$ have small $\UW_1$.

Gromov and Lawson show that in a closed 3-manifold of scalar curvature $\geq 1$ any homologically trivial curve bounds in its $2\pi$ neighborhood~\cite[Theorem 10.7]{GL83} and show that this property implies $\UW_1\leq 12\pi$ ~\cite[Corollary 10.11]{GL83}, but the proof seemingly requires some loops to be contractible, so it only works for, say, simply-connected manifolds. Positive answer to Question~\ref{main} would remove this limitation (at a cost of a constant factor). However, a uniform bound on $\UW_1$ of such manifolds has been already proven by other means in~\cite{LM20} (follows from Theorem 1.1 (b)).

We now state our results. Our first theorem generalizes and gives a different proof of the aforementioned result of Katz~\cite{Ka88}.

\begin{thmintro} \label{introthm:betti}
Let $X$ be a compact Riemannian polyhedron, and let $\widetilde{X}$ be the universal cover of $X$. Suppose $\pi_1(X)$ is virtually cyclic. Then we have
\[\UW_1(X) \leq 6 \cdot \UW_1(\widetilde{X}).\]
\end{thmintro}

For surfaces, we obtain a better upper bound for $\UW_1$ without any assumption on the fundamental groups.

\begin{thmintro}\label{thm:surface}
Let $\Sigma$ be a compact surface with a Riemannian metric and let $\widetilde{\Sigma}$ be the universal cover of $\Sigma$. Then we have 
\[\UW_1(\Sigma) \leq \UW_1(\widetilde{\Sigma}).
\]
\end{thmintro}

One reason why a positive answer to Question~\ref{main} in general might be challenging is that slight modifications on the question have negative answers. 
For instance, the answer to the Question~\ref{main} is negative if we replace $\UW_1$ with $\UW_2$.
There exist $4$-dimensional Riemannian manifolds $M$ and corresponding universal covers $\widetilde{M}$ such that $\UW_2(\widetilde{M})\leq 1$, yet $\UW_2(M)$ can be made arbitrarily large~\cite{ABG21}.
The same construction, adapted to lower dimensions, provides examples of closed Riemannian surfaces $M$ and corresponding covers $\widehat{M}$ such that $\UW_1(\widehat{M}) \leq 1$, while $\UW_1(M)$ can be made arbitrarily large. We now describe such examples.

\begin{example}\label{e:example}
In $\mathbb{R}^3$, take the standard cubic grid with the vertices in $\mathbb{Z}^3$. Let $Z$ be the one-dimensional skeleton of this grid, and let $Z^\vee$ be the one-dimensional skeleton of the dual grid, that is, $Z^\vee = Z + (\frac12, \frac12, \frac12)$. Let $\widehat{M}$ consist of the points
equidistant from $Z$ and $Z^\vee$. After a slight smoothing, $\widehat{M}$ becomes a Riemannian surface.  
The manifold $M$ is defined as the quotient of $\widehat{M}$ by the lattice $\Lambda$ generated by the following vectors: $v_1 = (R,0,0)$, $v_2 = (0,R,0)$, $v_3 = (\frac12,\frac12,\frac12+R)$ where $R$ is an integer. These three translations preserve $\widehat{M}$, and the last one swaps $Z$ and $Z^\vee$. Gromov's ``fiber contraction'' argument can be used to show that $\UW_1(M)$ is of order $R$. Namely, one can apply \cite[Corollary~2.3]{BB21} to the evident inclusion of $M$ in the $3$-dimensional torus
$\mathbb{R}^3 / \Lambda$; it is nontrivial at the level of $H_2(\cdot;\mathbb{Z}_2)$, 
since a generic circle in the torus parallel to $v_3$ intersects $M$ an odd number of times. 
Indeed, if a lift of such a circle to $\mathbb{R}^3$ begins closer to $Z$, 
it necessarily ends closer to $Z^\vee$.
Therefore, the width $\UW_1(M)$, sandwiched between the covexity radius of  $\mathbb{R}^3 / \Lambda$ and $\diam M$, is of order $R$ as well. As for the cover $\widehat{M}$, it can be projected to $Z$ with fibers of size $\approx 1$, so $\UW_1(\widehat M)$ is of order $1$.
\end{example}

We remark that $\widehat{M}$ in the above example is not the universal cover and therefore does not produce a negative answer of the Question~\ref{main}. In fact, the universal cover of a closed Riemannian surface is either a sphere or cannot have bounded Uryson 1-width. However, we show that if a negative answer to Question~\ref{main} exists, those examples can be found in the class of Riemannian $2$-complexes.
More precisely,

\begin{thmintro}\label{thm:low dim reduction}
Suppose that $\{X_n\}_{n=1}^\infty$ is a sequence of compact Riemannian polyhedra such that the ratio $\displaystyle\frac{\UW_1(X_n)}{\UW_1(\widetilde{X_n})}$ is unbounded and $\dim(X_n)$ is bounded. Then
\begin{itemize}
\item Then there exists $\{Z_n\}_{n=1}^\infty$ of Riemannian 2-complex such that the ratio $\displaystyle\frac{\UW_1(Z_n)}{\UW_1(\widetilde{Z_n})}$ is unbounded.
\item Then there exists a related sequence $\{Z_n\}_{n=1}^\infty$ of closed Riemannian 4-manifolds such that the ratio $\displaystyle\frac{\UW_1(Z_n)}{\UW_1(\widetilde{Z_n})}$ is unbounded.
\end{itemize}
\end{thmintro}

\begin{overview}
    The paper is organized as follows. In Section~\ref{s2}, we establish Theorem~\ref{introthm:betti}. Section~\ref{s3} is devoted to the proof of Theorem~\ref{thm:surface}, and in Section~\ref{s4}, we prove Theorem~\ref{thm:low dim reduction}.
\end{overview}

\begin{acknowledgements}
We thank Aleksandr Berdnikov for his valuable feedback on an earlier draft of the paper. We are also grateful to the referees for their detailed report and insightful comments, which greatly improved the paper. In particular, they identified a gap in the proof of Proposition 4.6, which has now been corrected.
\end{acknowledgements}

\section{spaces with virtually cyclic fundamental groups}\label{s2}

In this section our goal is to prove Theorem~\ref{introthm:betti}, which deals with
the case when $\pi _1(X)$ is virtually cyclic. We remark that when $\UW_1(\widetilde{X})$ is bounded, $\widetilde{X}$ is quasi-isometric to a tree. Moreover, if $X$ is compact, then $\pi_1(X)$ is quasi-isometric to $\widetilde{X}$ and hence quasi-isometric to a tree. 
So by Stallings' theorem on groups with more than one end \cite{stallings1971group}, $\pi_1(X)$ admits a non-trivial splitting as an amalgamated product or an HNN extension over some finite group. We note that the property of being finitely presented is invariant under quasi-isometries \cite[Ch. I, 8.24]{B-H}, so $\widetilde{X}$ is finitely presented as it is quasi-isometric to a free group. We can now keep splitting $\pi_1(X)$ and by Dunwoody's accessibility theorem for finitely presented groups \cite{Dunwoody} this process of splitting over finite groups eventually terminates.
In the terminal decomposition, $\pi_1(X)$ becomes the fundamental group of a finite graph of groups with finite vertex and edge groups.
By Bass--Serre theory, such a group always contains a free group of finite index.

The proof idea for Theorem~\ref{introthm:betti} comes from~\cite[Corollary 10.11]{GL83}. 
Fix $x_0\in X$ and consider the spheres around $x_0$, $S_r=S(x_0,r)$.
Without loss of generality, we may assume that $X$ has a piecewise flat metric.
We define a continuous map $g:X\to Y$ where $Y$ is a graph, by mapping
connected components of $S_r$ to points (formally $Y$ is a quotient
space of $X$ defined by declaring that two points of $X$ are equivalent
if and only if they lie in the same connected component of $S_r$ for some $r\geq 0$). 
To prove the Theorem~\ref{introthm:betti}, we need to show that the diameter of a connected component of
$S_r$ is small when $\UW_1(\widetilde{X})$ is small. 

To this end, we need the next two lemmas. The first lemma provides a condition under which the diameter of a connected component $C$ of $S_r$ is small. However, in applying this lemma to the proof of Theorem~A, we also require these components to be path connected, which need not always be the case. To address this, we instead formulate the lemma for a small neighborhood of $C$, which is always path connected.

\begin{lemma}\label{l:thin triangle to diam}
    Let $x_0,a,b\in X$ where $a,b$ both live inside a $\delta$-neighborhood $N_\delta(C)$ of a connected component $C$ of $S_r=S(x_0,r)$ for some $r$. Suppose $s_1=[x_0,a]$ and $s_2=[x_0,b]$.
    Suppose $x_1\in s_1$, $x_2\in s_2$ and $x_3\in N_\delta(C)$ such that $d(x_i,x_j)\leq \varepsilon$ for all $i,j$. Then $d(a,b)\leq 3\varepsilon+4\delta$. (See figure~\ref{f:thin to diam}.)
\end{lemma}

\begin{figure}[htp]
\begin{center}
\begin{tikzpicture}
    \draw (5,0) arc (0:30:5cm);
    \draw (0,0) -- (5,0);
    \draw (0,0) --(4.33012701892,2.5);
\node[circle, draw, fill=black, inner sep=2pt, label=$\quad a$]  at (5,0) {};
\node[circle, draw, fill=black, inner sep=2pt, label=$b$]  at (4.33012701892,2.5) {};
\node[circle, draw, fill=black, inner sep=2pt, label=$x_0$]  at (0,0) {};
\node[circle, draw, fill=black, inner sep=2pt, label=$x_1 \qquad$]  at (4,0) {};
\node[circle, draw, fill=black, inner sep=2pt, label=$x_2$]  at (3.46410161514, 2) {};
\node[circle, draw, fill=black, inner sep=2pt, label=$\quad x_3$]  at (4.82962913145,1.29409522551) {};
\draw[black,thick,dashed]  (3.46410161514, 2) -- (4,0);
\draw[black,thick,dashed]  (3.46410161514, 2) -- (4.82962913145,1.29409522551);
\draw[black,thick,dashed] (4.82962913145,1.29409522551) -- (4,0);
\end{tikzpicture}
\caption{If $x_1,x_2,x_3$ are close to each other then $a$ and $b$ will also be close to each other.}\label{f:thin to diam}
\end{center}
\end{figure}
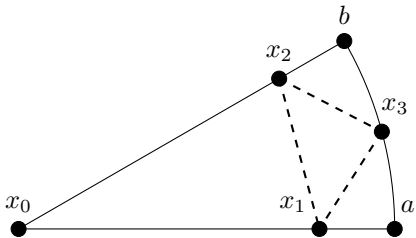
\begin{proof}
    Note that
   \begin{align*}
d(x_1,a)&= d(x_0,a)-d(x_0,x_1) \\
&\leq d(x_0,x_3)+2\delta-d(x_0,x_1)\\
&\leq d(x_1,x_3)+2\delta\\
&\leq \varepsilon+2\delta
   \end{align*}
   Similarly, $d(x_2,b)\leq \varepsilon+2 \delta$. Therefore
 \begin{align*}
d(a,b)&\leq d(a,x_1)+d(x_1,x_2)+d(x_2,b)\\
&\leq (\varepsilon+2\delta)+\varepsilon+(\varepsilon+2\delta)\\
&\leq 3\varepsilon +4\delta
\end{align*}
 \end{proof}

As illustrated in Figure~\ref{f:thin to diam}, in order to apply Lemma~\ref{l:thin triangle to diam}, we require three sides of specific triangles to contain points that are mutually close to one another. 
 This is where we are going to need the hypothesis that $\UW_1(\widetilde{X})$ is small.
That means there is a map from $\widetilde{X}$ to a tree that has small fibers.
The next lemma is a generalization of the statement that any map from a triangle to a tree has a fiber that intersects all three sides of the triangle.

\begin{lemma}\label{l:loops to tree}
    Let $f:P\rightarrow T$ be a map from an $n$-gon $P$ to a tree $T$ and $n\geq 3$.  Then there exists a fiber of the map $f$ that intersects three consecutive edges of $P$.
\end{lemma}

\begin{proof}
We prove the claim by induction on $n$.

For $n=3$, take a triangle $P$ with vertices $v_0,v_1,v_2$ and edges $e_0=[v_0,v_1]$, $e_1=[v_1,v_2]$, and $e_2=[v_2,v_0]$. 
Let $f:P\rightarrow T$ be a map to a tree.
Since $T$ is a tree, any path between two points in $T$ has to contain the unique geodesic between those points.
In particular, $[f(v_0),f(v_2)]\subset f(e_0\cup e_1)$.
Since $[f(v_0),f(v_2)]$ is connected, $f(e_0)\cap [f(v_0),f(v_2)]$ and $f(e_1)\cap [f(v_0),f(v_2)]$ intersect because they cover $[f(v_0),f(v_2)]$.
In other words, there exist $x\in e_0$ and $y\in e_1$ such that $f(x)=f(y)$ and $ f(x)\in[f(v_0),f(v_2)]$.
Since $[f(v_0),f(v_2)]\subset f(e_2)$ and $f(x)\in [f(v_0),f(v_2)]$, there exists $z\in e_3$ such that $f(z)=f(x)$.
The claim follows.

Suppose that the claim is true for any $n$-gon for some $n\geq 3$.
We consider a map $f:P\rightarrow T$ where $P$ is an $(n+1)$-gon and $T$ is a tree.
Choose two consecutive edges of $P$, and treating their union as a single edge, we obtain a new polygon 
$P'$, with $n$ sides.
By the induction hypothesis, there exists a fiber of the induced map $f:P'\rightarrow T$ that intersects three consecutive edges of $P'$. 
It follows that, either that fiber intersects three consecutive edges of $P$, or the fiber intersects the first and the third edges of three consecutive edges in $P$. In the second case, suppose $[v_0,v_1], [v_1,v_2], [v_2,v_3]$ are those three consecutive edges. Suppose $x\in [v_0,v_1]$ and $y\in [v_2,v_3]$ such that $f(x)=f(y)$. Since $f([x,v_1]), f([v_1,v_2])$ and $f([v_2,y])$ form a triangle in the tree $T$,
\[
f([x,v_1])\cap f([v_1,v_2])\cap f([v_2,y])\neq \emptyset.
\]
In particular, there exists a fiber that intersects $[v_0,v_1], [v_1,v_2]$, and $[v_2,v_3]$.
This completes the proof.
\end{proof}

\begin{proof}[Proof of Theorem~\ref{introthm:betti}]
We treat first the case when $\pi_1(X)$ is finite (done already in \cite[Theorem 3.1] {Ka88}, but we give a different proof).

Let $\phi:\widetilde{X}\to T$ be a continuous proper map to a tree $T$
such that the preimage of any point has diameter at most 1. We note that
we may take $T$ to be a tree rather than a graph as $\pi_1(\widetilde{X})=1$.

It suffices to show that every connected component of every sphere around $x$ has diameter at most $3$.  We take two points $a$ and $b$ from such a connected component $C$ of $S$.
Let $\gamma$ be a path between $a$ and $b$ in a $\delta$-neighborhood of $C$.
We use the $\delta$-neighborhood rather than $C$ itself because $C$ may not be path connected.
Let $s_1,s_2$ be geodesic paths parametrized by arc length joining $x$ to $a,b$ respectively.

We consider the loop $\alpha =s _1\cdot \gamma\cdot s _{2}^{-1}$. Since $\pi _1(X)$ is finite, for some $n$, $\alpha ^n$ is homotopically trivial so it lifts to a loop $\widetilde{\alpha ^n}$ in $\widetilde{X}$. 
We note that $\widetilde{\alpha ^n}$ is formed
by concatenating liftings of $s_1,\gamma, s_{2}^{-1}$
in this order. 

By Lemma~\ref{l:loops to tree}, there exists a  fiber of the restriction map $\phi:\widetilde{\alpha^n}\rightarrow T$ that intersects lifts of $s_1,\gamma,$ and $s_2^{-1}$.
Since fibers of $\phi$ have diameter at most 1, we have that there exists $x_1\in s_1, x_2\in s_2$ and $x_3\in \gamma$,
such that $d(\widetilde{x_i}, \widetilde{x_j}) \leq 1$ and by projecting the corresponding geodesic, this implies $d(x_i, x_j) \leq 1$.
By Lemma~\ref{l:thin triangle to diam}, we obtain that $d(a,b)\leq 3+4\delta$.
Since $a,b$ is taken arbitrarily from $C$ and $\delta$ can be taken arbitrarily close $0$, it follows that $\diam(S)\leq 3$.
\newline

 We now give a similar argument in the case when $\pi_1(X)$ is virtually cyclic. In fact, the same reasoning applies when $\pi_1(X)$ is virtually abelian. 
However, as mentioned at the beginning of this section, the assumptions 
$\UW_1(\widetilde{X}) < \infty$ and $\pi_1(X)$ virtually abelian together imply 
that $\pi_1(X)$ must already be virtually cyclic.

It suffices to show that every connected component of every sphere around $x$ has diameter at most $6$.  Suppose, on the contrary, that there exists a connected component $C$ with diameter $> 6$.
Specifically, there exists $a,c\in C$, such that $d(a,c)> 6$ with a path $\gamma$ between them that stays in some $\delta$-neighborhood of $C$. 
There is a point $b$ in $\gamma $ that is at distance $> 3$ from
both $a,c$. Let $\gamma =\gamma _1\cup \gamma _2$ where $a,b$, and $b,c$
are the endpoints of $\gamma _1$ and  $\gamma _2$ respectively.
Let $s_1,s_2,s_3$ be geodesic paths parametrized by arc length joining $x$ to $a,b,c$ respectively.

We consider the loops $\alpha =s_1\cdot \gamma _1\cdot s_2^{-1}$,
$\beta =s_2\cdot \gamma _2\cdot s_3^{-1}$.

Let $G$ be the finite-index cyclic subgroup of $\pi_{1}(X)$.  
Then there exist $m \in \mathbb{N}$ such that $\alpha^{m},\beta^{m} \in G$.  
Since $G$ is abelian, the commutator loop $\eta := \alpha^{m}\beta^{m}\alpha^{-m}\beta^{-m}$
represents the trivial element of $G$. In particular, $\eta$ lifts to a loop in the covering space $\widetilde{X}$.
Observe that, because $m>0$, the subpath $s_{2}^{-1}\cdot s_{2}$  appears once in $\eta$ and this is the only instance of backtracking in $\eta$. Removing this subpath from $\eta$, we obtain a homotopically trivial loop in $X$ expressed as a concatenation of paths from $V$, such that every three consecutive paths (up to orientation) are of one of the following types, in any order: $\{s_{i}, \gamma, s_{j}\},
\{s_{i}, \gamma_{1}, s_{j}\}, \text{or }
\{s_{i}, \gamma_{2}, s_{j}\}$ where $i \neq j.$
We now apply Lemma~\ref{l:loops to tree} to the restriction of $\phi$ on the lift of this loop in $\widetilde{X}$.
It follows that there exists a fiber that intersects the lifts of all three paths (taken upto orientation) in the set $\{s_i,\gamma,s_j\}$ for some $i\neq j$.

In each case, we project that fiber onto $X$. Since each fiber of $\phi$ has diameter $\leq 1$, its projection in $X$ also has diameter $\leq 1$.
In the first case, this gives us $x_1\in s_1,x_2\in \gamma, x_3\in s_2$ such that $d(x_i,x_j)\leq 1$ for all $i,j$.
By Lemma~\ref{l:thin triangle to diam}, it follows that $d(a,b)\leq 3+4\delta$.
Similarly, the second and the third case give us $d(b,c)\leq 3+4\delta$ and $d(a,c)\leq 3+4\delta$, respectively. 
Since $\delta$ can be chosen arbitrarily close to $0$, each case gives us a contradiction. 
\end{proof}

\begin{remark}
Our initial strategy for addressing Question~\ref{main} was to consider the following weaker formulation.
\begin{question}\label{pmain}
Does there exist a function $f:[0,\infty)\to [0,\infty)$ such that for any compact Riemannian  polyhedron $X$, we have 
\[\UW_1(X)\leq f(\dim H_1(X;\mathbb Q))\cdot\UW_1(\widetilde{X})?
\]
\end{question}
The proof of Theorem~\ref{introthm:betti} exploits the fact that many loops in $X$ are homotopically trivial when $\pi_1(X)$ is virtually cylic and hence lift to loops in $\widetilde{X}$.
However, our method fails to apply when homotopically trivial loops are much harder to find in $\pi_1(X)$. 
For instance, the answer to Question~\ref{pmain} remains unknown when the fundamental group is the free group on two generators.
\begin{question}
    Does there exist a constant $c>0$ such that for any compact Riemannian  polyhedron $X$ where $\pi_1(X)$ is the free group on $2$-generators, we have 
$\UW_1(X)\leq c\cdot\UW_1(\widetilde{X})?$
\end{question}

\end{remark}

\section{Surfaces}\label{s3}

The purpose of this section is to prove Theorem~\ref{thm:surface}.  Note that even though the theorem statement is for a surface with a smooth Riemannian metric, it also implies the analogous statement for a piecewise linear surface with a piecewise Euclidean metric, because such a piecewise linear surface $\Sigma$ can be approximated by a smooth surface in a way that changes distances by an arbitrarily small amount, which causes both $\UW_1(\Sigma)$ and $\UW_1(\widetilde{\Sigma})$ to change by an arbitrarily small amount.  In the proof of Theorem~\ref{thm:surface} we sometimes cut and paste in a way that does not preserve smoothness; however, after any such operation we can locally replace the metric by a smooth approximation.  For simplicity of reading, we do not mention these smooth approximations in the proof of Theorem~\ref{thm:surface}.

The proof of Theorem~\ref{thm:surface} is a bit long, and it may seem unnecessarily complicated if our main goal is to prove the conclusion up to a constant factor.  Thus, we state the following weaker version, and sketch a completely different proof method that may be more intuitive than the proof method we use to get the sharp constant in Theorem~\ref{thm:surface}.  After this proof sketch, the remainder of this section contains the proof of Theorem~\ref{thm:surface}.

\begin{theorem}[Weaker version of Theorem~\ref{thm:surface}]\label{thm:weak-surface}
Let $\Sigma$ be a compact orientable surface with boundary, with a Riemannian metric.  Then we have $\UW_1(\Sigma) \leq 5 \cdot \UW_1(\widetilde{\Sigma})$.
\end{theorem}

\begin{proof}[Proof sketch.]
Let $D > \UW_1(\widetilde{\Sigma})$ be arbitrary.  First we consider the easier case where every point in $\Sigma$ is within distance $D$ of $\del \Sigma$.  We define a deformation retraction of $\Sigma$ into a subset $\Gamma$ as follows.  Each point of $\del \Sigma$ starts moving at unit speed in the direction perpendicular to $\del \Sigma$.  When it first hits another such point, both points stop; informally, $\Gamma$ is the set of such stopping locations.  More precisely, $\Gamma$ is the closure of the set of points in $\Sigma$ that have more than one length-minimizing path to $\del \Sigma$; we can call $\Gamma$ the cut locus.  The deformation retraction is defined on all of $\Sigma$, not just on $\del \Sigma$: to find the image of an arbitrary point of $\Sigma$, not in $\Gamma$ or in $\del \Sigma$, we find the (unique) closest point in $\del \Sigma$ and follow the trajectory of that point to where it hits $\Gamma$.  If $\Gamma$ is $1$-dimensional, then we have successfully shown $\UW_1(\Sigma) \leq 2D$, because the set of trajectories arriving at each point of $\Gamma$ has diameter at most $2D$.  

The cut locus is not always a locally finite $1$-complex, so we may have to perturb slightly to get into the generic case where it is.  Specifically, there is a smooth map from $\del \Sigma \times [0, \infty)$ into $\Sigma$, given by sending $(p, t)$ to the point at distance $t$ along the geodesic from point $p$ perpendicular to $\del \Sigma$.  This map is not defined for all time, because the geodesic might run off the edge of the surface.  But it is smooth where it is defined.  The multi-jet transversality theorem (see~\cite[Theorem II.4.13]{GG73}) implies that if we define the cut locus in terms of a small $C^{\infty}$ perturbation of this map, we may assume that it consists of a finite set of edges, coming together at finitely many vertices.  (We do not include the full transversality details in this proof sketch.) 
Defining the deformation retraction in terms of this perturbed map, we obtain our desired conclusion in the special case where every point in $\Sigma$ is within distance $D$ of $\del \Sigma$.

In the general case, the rough picture, sketched in Figure~\ref{fig:weak-surface}, is that the set of points at greater than distance $D$ from $\del \Sigma$ forms a disjoint union of disks.  Because each disk lifts to $\widetilde{\Sigma}$ and $D > \UW_1(\widetilde{\Sigma})$, we have maps from these disks to various $1$-dimensional complexes with fibers of diameter at most $D$.  Away from these disks, we can use the deformation retraction strategy given above.  Where the two strategies meet, along the boundaries of the disks, the fibers from each strategy combine, so their diameters might get a little larger but not too much.

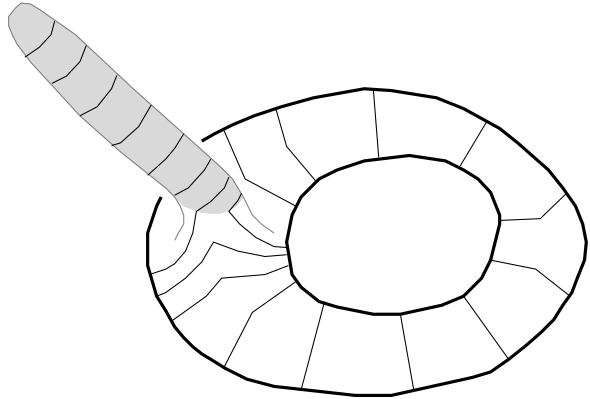
\begin{figure}
\begin{center}
\begin{tikzpicture}[scale=.12]
\draw[very thick] (-2, 10.2)--(-.6, 11)--(1.4, 12)--(3.8, 13)--(6.8, 14)--(10.4, 15)--(16, 16)--(19, 15.8)--(24, 15)--(27, 13.8)--(31, 11.6)--(33, 10)--(36.4, 7)--(38, 5)--(39.4, 3)--(40.2, 1)--(40.6, -1)--(40.4, -3)--(39.6, -5)--(39, -6.6)--(38, -8)--(37, -9.6)--(35.6, -11)--(34, -12.4)--(32, -14)--(30, -15.4)--(28, -16)--(19, -18)--(15, -18)--(6, -17)--(3, -16.2)--(.6, -15)--(-1, -14)--(-2, -13.4)--(-3, -12.6)--(-4, -11.6)--(-5, -10.4)--(-6, -8.6)--(-7, -7)--(-8, -3.6)--(-8, 0)--(-7, 3)--(-6.5, 4);
\draw[very thick] (8, 2)--(9, 4)--(11, 6)--(13, 7)--(16, 8)--(21, 8.6)--(25, 8)--(27, 7)--(28.6, 6)--(30, 4.5)--(31, 2)--(31, 1)--(30, -3)--(29, -5)--(27, -7)--(24.6, -8)--(20, -9)--(17, -9)--(13, -8.2)--(11, -7.6)--(9, -6)--(8, -4.6)--(7.4, -1)--cycle;

\draw[gray!30, fill=gray!30] (-5, 4)--(-6, 5)--(-7, 5.8)--(-9, 7.4)--(-11, 9)--(-15.5, 13)--(-21, 19)--(-22.5, 21)--(-23, 22)--(-23.4, 23)--(-23.4, 24)--(-22.6, 25)--(-22, 25.5)--(-21, 25.4)--(-20, 24.8)--(-16, 22)--(-9.6, 16)--(-4, 11)--(0, 7.4)--(1.4, 6)--(2, 5)--(2.6, 4)--(1.8, 3)--(1, 2.5)--(0, 2.2)--(-1, 2.2)--(-2, 2.3)--(-4, 3)--cycle;
\draw[gray] (-5, -.8)--(-4.6, 0)--(-4, 1)--(-4, 2)--(-4.4, 3)--(-5, 4)--(-6, 5)--(-7, 5.8)--(-9, 7.4)--(-11, 9)--(-15.5, 13)--(-21, 19)--(-22.5, 21)--(-23, 22)--(-23.4, 23)--(-23.4, 24)--(-22.6, 25)--(-22, 25.5)--(-21, 25.4)--(-20, 24.8)--(-16, 22)--(-9.6, 16)--(-4, 11)--(0, 7.4)--(1.4, 6)--(2, 5)--(2.6, 4)--(3.6, 2)--(4.6, 1)--(6, 0);

\draw (-7.8, -4.6)--(-6, -4)--(-5, -3.4)--(-3.8, -2)--(-3, 0)--(-2.6, 2.4)--(-1, 3.5)--(.5, 5)--(1, 6.2);
\draw (2.4, 4.4)--(2, 3.6)--(1, 2.4)--(2.2, 1)--(4, -.5)--(6, -1.5)--(7.4, -1.6);
\draw (-7, -7)--(-6, -6.6)--(-3.8, -5)--(-2, -3.2)--(-.7, -1)--(2, -2)--(5, -2.6)--(7.4, -2.4);
\draw (-5.4, -9.8)--(-1.5, -7)--(.2, -5)--(5, -4.6)--(7.6, -3.6);
\draw (.4, -15)--(3.6, -8.8)--(8.4, -5.4);
\draw (9, -17.4)--(11.6, -7.6);
\draw (21.5, -17.5)--(20, -9);
\draw (32, -14)--(27, -7);
\draw (38.8, -7)--(35, -4)--(30, -3);
\draw (38.4, 4.4)--(35.5, 1.6)--(31, 1.4);
\draw (26.6, 7.4)--(29.6, 12.4);
\draw (17.6, 8.4)--(17, 16);
\draw (10.6, 5.8)--(7.4, 9.6)--(6.2, 13.8);
\draw (8.4, 3)--(2.8, 6)--(.4, 11.5);

\draw (-5, 4.2)--(-3.5, 5)--(-1.7, 7)--(-1, 8.2);
\draw (-8, 6.4)--(-6, 8.2)--(-5, 9.4)--(-4, 11);
\draw (-12, 9.7)--(-11, 10)--(-9, 11.8)--(-7.6, 14.2);
\draw (-15.5, 13)--(-13.6, 14)--(-12, 16)--(-11.4, 17.5);
\draw (-18.6, 16.6)--(-17, 17.4)--(-15.5, 19)--(-14.8, 20.8);
\draw (-21.6, 19.5)--(-20, 20.6)--(-18.7, 22)--(-18.3, 23.5);
\end{tikzpicture}
\end{center}
\caption{Each fiber of the map in Theorem~\ref{thm:weak-surface} is either far from the boundary and roughly parallel to it (gray region), or near the boundary and orthogonal to it, or a union of one piece of the first type with some pieces of the second type.}\label{fig:weak-surface}
\end{figure}

To be more precise, for each $r \geq 0$, let $E_r$ be the set of points in $\Sigma$ at distance at least $r$ from $\del \Sigma$, and let $\widetilde{E}_r$ be the preimage of $E_r$ in $\widetilde{\Sigma}$.  Let $f \co \widetilde{\Sigma} \rightarrow Y$ be a map to a $1$-dimensional simplicial complex $Y$, with fibers of diameter at most $D$, and without loss of generality we may assume that $Y$ is simply connected, that is, a tree.

Under the map $f$, the image of $\widetilde{E}_{2D}$ is disjoint from the image of $\del \widetilde{\Sigma}$.  Thus, there is a disjoint union of fibers, which each have diameter at most $D$, that separate these two sets from each other.  This implies that every component of $\widetilde{E}_{2D}$ is bounded, because there is no way to separate an unbounded subset of $\widetilde{\Sigma}$ from $\del \widetilde{\Sigma}$ using a bounded set: because $\widetilde{\Sigma}$ is uniformly contractible, any bounded subset of its interior is contained in a topological disk, and the complement of that disk is connected. 

Let $A \subseteq \widetilde{E}_{2D}$ be a set projecting homeomorphically to $E_{2D}$; that is, $A$ consists of one lift per connected component of $E_{2D}$.  Let $Y_A$ be the image of $A$ under $f$, and let $f_A \co E_{2D} \rightarrow Y_A$ denote the composition of the homeomorphism and $f$.  This will be our final map, except that the fibers of $\del E_{2D}$ will be combined with some additional points outside $E_{2D}$.

As in the case where all of $\Sigma$ is within $D$ of $\del \Sigma$, consider the deformation retraction to the cut locus $\Gamma$, after a perturbation that ensures that $\Gamma$ is a $1$-complex. We can assume that for a generic metric $\del E_{2D} \cup \Gamma$ is a locally finite $1$-complex.
Now, in the complement of $E_{2D}$ if we push each point away from $\del \Sigma$ until it meets either $\Gamma$ or $\del E_{2D}$, we obtain a deformation retraction from $\Sigma$ to $E_{2D} \cup \Gamma$, with the fiber at each point interior to $E_{2D}$ being just that point, and the fiber at each point of $\del E_{2D} \cup (\Gamma\setminus E_{2D})$ being the union of one or more geodesics through that point, each of length at most $2D$.

Roughly, we want to compose this deformation retraction with $f_A$.   
More precisely, we should form a $1$-dimensional complex $Y'_A$ by taking each connected component of $\Gamma \setminus E_{2D}$, and attaching each boundary point in $\Gamma \cap \del E_{2D}$ to its image in $Y_A$.  Then there is a map $f'_A \co \Sigma \rightarrow Y'_A$, given by first deformation retracting to $E_{2D} \cup \Gamma$, and then applying $f_A$ to the points of $E_{2D}$ while applying the obvious identification map on the points of $\Gamma \setminus E_{2D}$.  

We claim that the fibers of the resulting map $f'_A$ have diameter at most $2D + D + 2D = 5D$.  In each fiber over a point of $Y_A$, any two points are each within $2D$ of their destinations under the deformation retraction, and those points in $E_{2D}$ are within distance $D$ of each other, because the corresponding points in $A$ are within distance $D$ of each other and the projection map does not increase distance.  
And, in each fiber over a point of $\Gamma \setminus E_{2D}$, any two points have the same destination under the deformation retraction, so they are within $2D + 2D = 4D$ of each other.  Thus we may conclude $\UW_1(\Sigma) \leq 5\cdot \UW_1(\widetilde{\Sigma})$, as desired.
\end{proof}

In the proof of Theorem~\ref{thm:surface}, it is inconvenient to estimate $\UW_1(\Sigma)$ in terms of its standard definition, because we end up repeatedly modifying the $1$-dimensional target space to which $\Sigma$ maps.  Thus, we use the perspective of separating sets, introduced in~\cite{Papasoglu20}.  Let $X$ be a $2$-dimensional Riemannian polyhedron.  We say that a $1$-dimensional subpolyhedron $Z$ is a \textbf{\textit{$D$-separator}} of $X$ if every path component of $Z$ and every path component of $X \setminus Z$ has diameter at most $D$.  
\begin{lemma}\label{lem:width and separator}
    Uryson 1-width of $X$ is the infimal $D$ such that $X$ admits a $D$-separator.
\end{lemma}
\begin{proof}  
    If $X$ admits a $D$-separator, then we can map each path component of $Z$ to a point, and for each path component of $X \setminus Z$, we can map it to the cone over the points corresponding to the components of $Z$ in its boundary.  In the reverse direction, if $X$ admits a map to a $1$-dimensional space $Y$ with fibers of diameter at most $D$, then for any $\varepsilon > 0$, by finely subdividing $Y$ we may assume that the preimage of each vertex or edge of $Y$ has diameter at most $D+\varepsilon$, and then take $Z_0$ to be the preimage of the vertices of $Y$.
If some components of $Z_0$ are not $1$-dimensional, replace each such component by the boundary of a small regular neighborhood in $X$.
The resulting set $Z$ is a $(D+2\varepsilon)$-separator, and letting $\varepsilon \to 0$ proves the claim.
\end{proof}

According to Lemma~\ref{lem:width and separator}, we may express the proof of Theorem~\ref{thm:surface} entirely in terms of $D$-separators, and never refer to a $1$-dimensional target space $Y$.

We divide the proof of Theorem~\ref{thm:surface} into two cases. The first case concerns surfaces with boundary, and the second addresses all the remaining cases. 
We say a path in a surface $\Sigma$ with endpoints in $\partial \Sigma$ is a \emph{\textbf{nontrivial cutting}} if it is impossible to homotope into $\del \Sigma$ with endpoints fixed.

The proof for surfaces with boundary relies on repeated applications of the following proposition.

\begin{proposition}[Main proposition for surfaces with boundary]\label{lem:surface}
Let $\Sigma$ be a compact surface with boundary, with a Riemannian metric.  
Let $\gamma$ be a path in $\Sigma$ which is length-minimizing among all the nontrivial cutting paths.
Let $L$ be the length of $\gamma$, and let $\widetilde{\gamma}$ be a lift of $\gamma$ to $\widetilde{\Sigma}$.  For any $M > 0$, let $\widetilde{\Sigma}'$ be the result of cutting apart $\widetilde{\Sigma}$ along $\widetilde{\gamma}$ and gluing in a Euclidean strip $[-M, M] \times [0, L]$, such that its ends $\{-M\} \times [0, L]$ and $\{M\} \times [0, L]$ attach isometrically to the two cut copies of $\widetilde{\gamma}$.  
Suppose that $Z$ is a $D$-separator in $\widetilde{\Sigma}$ for some $D > 0$, containing only finitely many points of $\widetilde{\gamma}$.  

Then for all $\varepsilon$ with $0 < \varepsilon < M$ there exists a $D+\varepsilon$-separator $Z'$ of $\widetilde{\Sigma}'$, with the following properties:
\begin{enumerate}
    \item $Z$ and $Z'$ agree on the complement of the added strip; and
    \item The segments $\{-M+\varepsilon\} \times [0, L]$ and $\{M-\varepsilon\} \times [0, L]$ are in $Z'$, and within $(-M+\varepsilon, M-\varepsilon)\times [0, L]$, the separator $Z'$ consists only of segments $\{x\} \times [0, L]$ for various $x$.
\end{enumerate}
\end{proposition}

First we prove Theorem~\ref{thm:surface} for the surface with boundary case assuming Proposition~\ref{lem:surface}.  Then we prove Lemmas~\ref{lem:fill-in}, \ref{lem:fill-out}, and~\ref{lem:draw-z} to help prove Proposition~\ref{lem:surface}.  To replace $Z$ by $Z'$ in Proposition~\ref{lem:surface}, the construction essentially consists of cutting each component of $Z$ or its complement into a left half and a right half, and adding to each half a new piece that follows the vertical edge of the strip.  However, it needs to be done carefully in order to ensure that the additions do not combine any components of $Z$ or its complement, and (stated informally) to ensure that if an addition to the left half includes faraway points, then the right half also contained equally faraway points, and vice versa.  Lemmas~\ref{lem:fill-in}, \ref{lem:fill-out}, and~\ref{lem:draw-z} allow us to make these verifications.

\begin{proof}[Proof of Theorem~\ref{thm:surface} for surface with boundary]
If $\Sigma$ is a disk, then $\widetilde{\Sigma} = \Sigma$ and the theorem is a tautology.  Thus, we may assume that $\pi_1(\Sigma)$ is a free group. 

The rough idea of the proof is as follows. 
Suppose that there are $r$ disjoint length-minimizing geodesics in $\Sigma$, such that cutting $\Sigma$ along these geodesics results in a fundamental domain for $\Sigma$ in $\widetilde{\Sigma}$ homeomorphic to a disk.
We know that $\widetilde{\Sigma}$ admits a $D$-separator $Z$ for some $D$ close to $\UW_1(\widetilde{\Sigma})$.  Suppose that we can apply Proposition~\ref{lem:surface} to all the lifts of all these geodesics at once, and consider the resulting $D+\varepsilon$-separator $Z'$.  We can cut along the geodesics to obtain a fundamental domain, glue corresponding pairs of geodesics, and shrink the strips to recover $\Sigma$.  This process produces a $(D+\varepsilon)$-separator of $\Sigma$, using the restriction of $Z'$ to the fundamental domain, so $\UW_1(\Sigma)$ is no larger than $D+\varepsilon$, which is arbitrarily close to $\UW_1(\widetilde{\Sigma})$.

The precise version of the proof requires choosing these geodesics one at a time, 
so that the successive applications of Proposition~\ref{lem:surface} do not interfere 
with each other. 
Moreover, in this refined version of the argument, cutting along 
these geodesics may not yield a single fundamental domain, but rather a disjoint 
union of disks, which is nevertheless sufficient for the argument to proceed.
Let $M$ be a large number, larger than $\UW_1(\widetilde{\Sigma})+2r\varepsilon$. (This specific threshold is explained later in the proof, but at this stage, we just let $M$ be sufficiently large.)  
Let $\Sigma_0 = \Sigma$, and let $\gamma_1$ be a path in $\Sigma_0$ with endpoints in $\del \Sigma_0$, impossible to homotope into $\del \Sigma_0$ with endpoints fixed, and length-minimizing among such paths.  
Equivalently, every lift $\widetilde{\gamma_1}$ is length-minimizing among paths in $\widetilde{\Sigma_0}$ that connect two distinct connected components of $\del \widetilde{\Sigma_0}$.  Let $L_1$ be the length of $\gamma_1$.

Let $\Sigma_1$ be the result of cutting $\Sigma_0$ along $\gamma_1$ and gluing in a Euclidean strip $[-M, M] \times [0, L_1]$, such that its ends $\{-M\} \times [0, L_1]$ and $\{M\} \times [0, L_1]$ attach isometrically to the two cut copies of $\gamma_1$.  Abusing notation, in $\Sigma_1$ we let $\gamma_1$ denote the middle segment $\{0\} \times [0, L_1]$.

To find $\gamma_2$ in $\Sigma_1$, we let $\gamma_2$ be a length-minimizing nontrivial cutting path in the surface resulting from cutting $\Sigma_1$ along $\gamma_1$. 
We note that the endpoints of $\gamma_2$ are not along the two cut copies of $\gamma_1$; this is because if an endpoint of $\gamma_2$ were in $\gamma_1$, then the length of $\gamma_2$ would have to be greater than $M$.    
Moreover, the endpoints of $\gamma_2$ are not along the added strip in $\Sigma_1$, because otherwise $\gamma_2$ would be orthogonal to boundary. Since the added strip is a flat rectangle and $\gamma_2$ is disjoint from the side $\gamma_1$ of the rectangle, this  means that $\gamma_2$ would be parallel to $\gamma_1$, which is excluded.

Gluing the two copies of $\gamma_1$ together again, we have $\Sigma_1$ with both $\gamma_1$ and $\gamma_2$ inside as disjoint geodesics.  Let $L_2$ be the length of $\gamma_2$, let $\Sigma_2$ be the result of cutting $\Sigma_1$ along $\gamma_2$ and gluing in $[-M, M] \times [0, L_2]$, and let $\gamma_2$ also denote the middle segment $\{0\} \times [0, L_2]$ in $\Sigma_2$.

We repeat this process to get $\gamma_3, \ldots, \gamma_r$ and $\Sigma_3, \ldots, \Sigma_r$.  Each time, we select $\gamma_i$ to be a length-minimizing nontrivial cutting path in the surface $\Sigma'_{i-1}$ that results from cutting $\Sigma_{i-1}$ along $\gamma_1, \ldots, \gamma_{i-1}$, which are disjoint segments each in the center of a strip.  The resulting $\gamma_i$ is disjoint from $\gamma_1, \ldots, \gamma_{i-1}$.  
Let $L_i$ be the length of $\gamma_i$, let $\Sigma_i$ be the result of cutting $\Sigma_{i-1}$ along $\gamma_i$ and gluing in $[-M, M] \times [0, L_i]$, and let $\gamma_i$ also denote the middle segment $\{0\} \times [0, L_i]$ in $\Sigma_i$.
Since each $\gamma_i$ is a nontrivial cutting path contained in some connected component of $\Sigma'_{i-1}$, cutting that component along $\gamma_i$ either produces a single connected component with a smaller first Betti number, or it produces two connected components, with nonzero first Betti numbers summing to the first Betti number of $\Sigma'_{i-1}$.  In either case, after cutting, each resulting connected component has a smaller first Betti number.
Hence, this process terminates after finitely many iterations because $\Sigma$ is compact.
In other words, for some $r<\infty$, $\Sigma'_r$ is a disjoint union of disks.
In $\Sigma_r$, each geodesic $\gamma_1, \ldots, \gamma_r$ is surrounded by a Euclidean strip of length $2M$ and these strips are pairwise disjoint.

At the end of the proof, we will use the fact that for any $\varepsilon > 0$, each $\Sigma_i$ admits a  homeomorphism to $\Sigma_{i-1}$ that increases distances by at most $\varepsilon$.
This is obtained by mapping the long strip around $\gamma_i$ in $\Sigma_i$ to a small tubular neighborhood around $\gamma_i$ in $\Sigma_{i-1}$.  The main effect of this process is to decrease lengths in the strip direction, but it may slightly increase some lengths in the direction parallel to $\gamma_i$.

Composing the maps, we obtain a homeomorphism from $\Sigma_r$ to $\Sigma_0 = \Sigma$ that increases the distances by at most $r\varepsilon$ amount.

We let $\Sigma'_0:=\Sigma$. We apply Proposition~\ref{lem:surface} to construct separators on $\widetilde{\Sigma'_1}, \ldots, \widetilde{\Sigma'_r}$, with an eye toward being able to modify the separator on $\widetilde{\Sigma'_r}$ to get a separator of $\Sigma_r$. 
Note that for any $i\geq 1$, $\Sigma_i'$ may have more than one component in which case $\widetilde{\Sigma'_i}$ would mean union of universal covers of each component of $\Sigma'_i$.
Specifically, we will construct $Z_0, \ldots, Z_r$, such that each $Z_i$ is a $\UW_1(\widetilde{\Sigma})+(2i+2)\varepsilon$-separator of $\widetilde{\Sigma'_i}$.  We know that $Z_0$ exists, because $\UW_1(\widetilde{\Sigma})+2\varepsilon > \UW_1(\widetilde{\Sigma'_0})$.

Suppose we have already constructed the separators 
$Z_0, Z_1, \ldots, Z_{i-1}$
on the corresponding surfaces 
$\widetilde{\Sigma'_0}, \widetilde{\Sigma'_1}, \ldots, \widetilde{\Sigma'_{i-1}}$.
We now proceed to construct $Z_i$ on $\widetilde{\Sigma'_i}$.

Assume that the curve \( \gamma_i \) lies in a connected component \( K \) of \( \Sigma'_{i-1} \).
Let \( K' \) denote the collection of components of \( \Sigma'_i \) that arise from \( K \).
All other components of \( \Sigma'_i \) are the same as those of \( \Sigma'_{i-1} \), and on their universal covers we simply keep the same separator \( Z_{i-1} \).
Hence, it remains to define \( Z_i \) on \( \widetilde{K'} \), the union of the universal covers of the components in \( K' \).

We already have the separator \( Z_{i-1} \) on \( \widetilde{K} \).
Since \( \gamma_i \) is length-minimizing among all nontrivial cutting paths in \( K \), we apply Proposition~\ref{lem:surface} to \( Z_{i-1} \) for every lift of \( \gamma_i \) to \( \widetilde{K} \).
Although there are infinitely many such lifts, the corresponding replacements can be performed in any order, since they are pairwise disjoint and do not interact with each other.
By construction, the resulting separator restricts to a separator on \( \widetilde{K'} \), and we define \( Z_i \) to be this separator on \( \widetilde{K'} \).

To estimate the diameters of the components of \( Z_i \) and its complement, we need to account for the distances between points lying in different added strips.
While performing a single replacement increases the diameters by at most \( \varepsilon \), performing multiple replacements increases them by at most \( 2\varepsilon \).
This is because any geodesic between points lying in two different added strips can be approximated by a geodesic in the original surface, together with two short paths—one in each added strip—each of length at most $\varepsilon$.
Therefore, assuming that $Z_{i-1}$ was a $(\UW_1(\widetilde{\Sigma}) + 2i\varepsilon)$-separator of $ \widetilde{\Sigma'_{i-1}} $, we conclude that $ Z_i $ is a $(\UW_1(\widetilde{\Sigma}) + (2i + 2)\varepsilon)$-separator of $ \widetilde{\Sigma'_i} $.

Note that one byproduct of Proposition~\ref{lem:surface} is the proof that each $L_i$ is less than or equal to $\UW_1(\widetilde{\Sigma})+2i\varepsilon$; thus, if our $M$ is larger than the threshold of $\UW_1(\widetilde{\Sigma})+2r\varepsilon$, it is large enough to ensure that the length-minimizing choices of $\gamma_1, \ldots, \gamma_r$ are disjoint.

We can use $Z_r$ to construct a separator $Z'$ of $\Sigma_r$.  Since each connected component of $\Sigma'_r$ is simply connected, $\widetilde{\Sigma'_r}=\Sigma_r'$ and hence $Z_r$ is a separator on $\Sigma_r'$.
Now we glue together the corresponding copies of each $\gamma_i$ in $\Sigma_r'$ to form $\Sigma_r$, and include $\gamma_1, \ldots, \gamma_r$ as part of $Z'$.  The resulting set $Z'$ is a $\UW_1(\widetilde{\Sigma})+(2r+2)\varepsilon$-separator of $\Sigma_r$.

Taking the image of $Z'$ under a  homeomorphism from $\Sigma_r$ to $\Sigma$ that increases distances by at most $r\varepsilon$, we obtain a $\UW_1(\widetilde{\Sigma})+(3r+2)\varepsilon$-separator of $\Sigma$.  Because $\varepsilon$ may be arbitrarily small, we conclude $\UW_1(\Sigma) \leq \UW_1(\widetilde{\Sigma})$, as desired.
\end{proof}

In the lemmas to prove Proposition~\ref{lem:surface}, we imagine $\widetilde{\gamma}$ running upward in $\widetilde{\Sigma}$.  From this perspective, $\widetilde{\Sigma}\setminus \widetilde{\gamma}$ has a left component and a right component, and the boundary components of $\widetilde{\Sigma}$ containing $\widetilde{\gamma}(0)$ and $\widetilde{\gamma}(L)$ are the bottom and top components, respectively, of $\del \widetilde{\Sigma}$.  We focus on the left end of the strip; the right end is analogous.  We need to show that for each component of $Z$ or of $\widetilde{\Sigma} \setminus Z$ that touches $\widetilde{\gamma}$, replacing the portion on the right side by an arc running `parallel' to 
$\widetilde{\gamma}$ and lying in the $\varepsilon$-neighborhood of $\widetilde{\gamma}$, increases diameter by at most $\varepsilon$.  The following two lemmas help us to estimate the distance from a point on the left side to various points along $\widetilde{\gamma}$, which in turn helps us to estimate the distance to various points in $[-M, -M+\varepsilon]\times [0, L]$.

Throughout the paper $B(x,D)$ denotes the open ball of radius $D$ centered at $x$.
\begin{lemma}\label{lem:fill-in}
Let $x$ be a point in the left component of $\widetilde{\Sigma} \setminus \widetilde{\gamma}$, and suppose that the ball $B(x, D)$ contains two points $a_1$ and $a_2$ on $\widetilde{\gamma}$, as well as a path from $a_1$ to $a_2$ in the right component of $\widetilde{\Sigma} \setminus \widetilde{\gamma}$.  Then $B(x, D)$ also contains the interval in $\widetilde{\gamma}$ between $a_1$ and $a_2$.
\end{lemma}

\begin{proof}
Let $\pi$ be the path from $a_1$ to $a_2$ in the right component of $\widetilde{\Sigma}\setminus \widetilde{\gamma}$.  We draw geodesics back to $x$ from every point of $\pi$.  Suppose to the contrary that there is a point $a\in \widetilde{\gamma}$ between $a_1$ and $a_2$ that is not in $B(x, D)$.  Then there is some point $b$ in $\pi$ such that there are two length minimizing geodesics from $x$ to $b$, one crossing $\widetilde{\gamma}$ above $a$ at a point $c_1$, and the other crossing $\widetilde{\gamma}$ below $a$ at a point $c_2$. We have
\begin{align*}
2D \leq 2\cdot d(x, a)  &\leq d(x, c_1) + d(c_1, a) + d(x, c_2) + d(c_2, a) \\
&= d(x, c_1) + d(c_1, c_2) + d(x, c_2) \\
&\leq d(x, c_1) + d(c_1, b) + d(x, c_2) + d(c_2, b) = 2 \cdot d(x, b) < 2D,
\end{align*}
giving a contradiction.
\end{proof}

\begin{lemma}\label{lem:fill-out}
Let $x$ be a point in the left component of $\widetilde{\Sigma} \setminus \widetilde{\gamma}$, and suppose that the ball $B(x, D)$ contains a point $b$, which is in the right component of $\widetilde{\Sigma} \setminus \widetilde{\gamma}$ and is in $\del \widetilde{\Sigma}$.  Then,
\begin{enumerate}
    \item 
If $b$ is in the bottom component, and $c$ is any 
point along $\widetilde{\gamma}$ such that $B(x,D)$ contains a path between $c$ and $b$ on the right component of $\widetilde{\Sigma}\setminus \widetilde{\gamma}$, then the segment from $c$ to
$\widetilde{\gamma}(0)$ is in $B(x, D)$.
\item If $b$ is in the top component, and $c$ is any 
point along $\widetilde{\gamma}$ such that $B(x,D)$ contains a path between $c$ and $b$ on the right component of $\widetilde{\Sigma}\setminus \widetilde{\gamma}$, then the segment from $c$ to
$\widetilde{\gamma}(L)$ is in $B(x, D)$. And,
\item If $b$ is in a third component of $\del \widetilde{\Sigma}$, then $B(x, D)$ also contains all of $\widetilde{\gamma}$.
\end{enumerate}
\end{lemma}

\begin{proof}
To prove statement~(1), it suffices, by Lemma~\ref{lem:fill-in}, to show that 
\( B(x, D) \) contains a path from \( b \) to \( \widetilde{\gamma}(0) \) lying in the right 
component of \( \widetilde{\Sigma} \setminus \widetilde{\gamma} \). Indeed, once such a path 
exists, the hypothesis guarantees the existence of a path in \( B(x, D) \) connecting 
\( c \) to \( b \) to \( \widetilde{\gamma}(0) \) within the same right component of 
\( \widetilde{\Sigma} \setminus \widetilde{\gamma} \). We can then apply 
Lemma~\ref{lem:fill-in} to obtain the desired conclusion.
To this end, let $a$ be the point where the geodesic $\pi$ from $x$ to $b$ crosses $\widetilde{\gamma}$.  From $a$, the distance to the bottom boundary component of $\del \widetilde{\Sigma}$ is achieved by following $\widetilde{\gamma}$, so the distance from $a$ to $\widetilde{\gamma}(0)$ is less than or equal to the distance from $a$ to $b$ along $\pi$.  Thus, the path from $x$ to $a$ to $\widetilde{\gamma}(0)$ has length at most the length of $\pi$, which is less than $D$. Therefore, this path lies in  $B(x, D)$, so does the path from $b$ to $a$ to $\widetilde{\gamma}(0)$.  This completes the proof of statement (1), and the proof of statement (2) is exactly analogous.

To prove statement (3), again the geodesic $\pi$ from $x$ to $b$ crosses $\widetilde{\gamma}$ at a point $a$.  The length of $\widetilde{\gamma}$ is less than or equal to the length from $\widetilde{\gamma}(0)$ to $a$ to $b$, as well as the length from $\widetilde{\gamma}(L)$ to $a$ to $b$, because $\widetilde{\gamma}$ has minimum length among paths connecting distinct components of $\del \widetilde{\Sigma}$.  Thus, the paths from $x$ to $a$ to each endpoint of $\widetilde{\gamma}$ have length less than $D$, so all of $\widetilde{\gamma}$ is in $B(x, D)$.
\end{proof}

We denote by $\{Z_\alpha\}$ the collection of path-connected components of the intersection of $Z$ with the right side of $\widetilde{\Sigma} \setminus \widetilde{\gamma}$, and by $\{U_\beta\}$ the collection of path connected components of the complement of $Z$ in the right side of $\widetilde{\Sigma} \setminus \widetilde{\gamma}$.

\begin{lemma}\label{l:Uj}
    \begin{enumerate}
\item Two points $a,b\in \widetilde{\gamma}$ are in the same component $U_j\in \{U_\beta\}$ iff there is no arc in some $Z_i\in \{Z_\alpha\}$ that connects  the open interval $(a,b)$ between $a$ and $b$ in $\widetilde{\gamma}$ to $\partial \widetilde{\Sigma}\cup \widetilde{\gamma}\setminus [a,b]$.

\item Let $p = \sup \{ t \mid \tilde{\gamma}(t) \in U_j \}$, and suppose $\tilde{\gamma}(p) \in Z_i$. Then $U_j$ intersects the top component of $\partial \tilde{\Sigma}$ if and only if no arc in $Z_i$ connects $\tilde{\gamma}(p)$ to either $\tilde{\gamma}([0, p))$ or to any boundary component of $\partial \tilde{\Sigma}$ on the right side of $\widetilde{\Sigma}\setminus \widetilde{\gamma}$ other than the top one.

Let $p = \inf \{ t \mid \tilde{\gamma}(t) \in U_j \}$, and suppose $\tilde{\gamma}(p) \in Z_i$. Then $U_j$ intersects the bottom component of $\partial \tilde{\Sigma}$ if and only if no arc in $Z_i$ connects $\tilde{\gamma}(p)$ to either $\tilde{\gamma}((p, L])$ or to any boundary component of $\partial \tilde{\Sigma}$ on the right side of $\widetilde{\Sigma}\setminus \widetilde{\gamma}$ other than the bottom one.
\end{enumerate}
\end{lemma}
\begin{proof}
\begin{enumerate}
    \item Let $D$ be the right side of $\widetilde{\Sigma}\setminus \widetilde{\gamma}$ including $\widetilde{\gamma}$.

($\Rightarrow$) Suppose there exists an arc $\ell\subset Z_i$ in $D$ connecting a point $c \in (a,b)$ to a point $d \in \partial D \setminus [a,b]$. Since $D$ is simply connected, applying the Mayer–Vietoris sequence to the complement of $\ell$ and a small neighborhood of $\ell$ we obtain that $D\setminus \ell$ has two connected components: one containing $a$  and the other containing $b$. Hence, $a$ and $b$  cannot belong to the same $U_j$.

($\Leftarrow$) For the converse, it is enough to show that if an arc has both of its endpoints in  $(a,b)$, then the arc cannot separate $a$ and $b$ in $D$. We can assume that the interior of the arc lives outside $\widetilde{\gamma}$. Then we can construct a simple closed loop by concatenating this arc with the geodesic joining the endpoints of the arc along $\widetilde{\gamma}$. Since $D$ is simply connected, this loop bounds a disk in $D$. By Mayer--Vietoris sequence applied to a small neighborhood of this disk and its complement, we have that the complement of this disk is path connected and both $a,b$ live in the complement of that disk.

\item We only prove the first part, the second part is analogous.

($\Rightarrow$) If there is an arc in $Z_i$ connecting $p$ to $\widetilde{\gamma}([0,p))$, say at $\widetilde{\gamma}(q)$, then by a similar argument as is part (1), $Z_i$ separates $\widetilde{\gamma}((q,p))$ from the top component. Since  $U_j$ intersects $\widetilde{\gamma}((q,p))$, it follows that $U_j$ does not intersect the top component.
Similarly, if there is an arc in $Z_i$ connecting $p$ to any other component other than the top component, then $Z_i$ separates $\widetilde{\gamma}((0,p))$ from the top component. Since  $U_j$ intersects $\widetilde{\gamma}((0,p))$, it follows that $U_j$ does not intersect the top component.

 ($\Leftarrow$) Let $x:=\widetilde{\gamma}(t)\in \widetilde{\gamma}([0,p))$ be a point in $U_j$ such that $\widetilde{\gamma}([t,p))\subset U_j$. For the converse, suppose there is no arc in $Z_i$ that connects $p$ to either $\widetilde{\gamma}([0,p))$, or any boundary components other than the top boundary component. That means arcs in $Z_i$ connects $p$ to either $\widetilde{\gamma}((p,L])$ or to the top component. In either case, arguing as the converse part of the part (1), we can show that $Z_i$ cannot separate  $x$ and the top boundary component. It remains to show that, no other $Z_j$ separates $x$ and the top component. 
On the contrary, if some $Z_j$ with $j\neq i$ does that, by a similar argument as in part (1), $Z_j$  connects $\widetilde{\gamma}((x,L))$ to either $\widetilde{\gamma} ([0,x))$ or a boundary component other than the top (and similar replacements in the next two sentences).
By our choice of $x$, it follows that $Z_j$ must connect  $\widetilde{\gamma}((p,L])$ to either $\widetilde{\gamma} ([0,x))$ or a boundary component other than the top component.
On the other hand, $Z_i$ separates $\widetilde{\gamma}((p,L])$ from $\widetilde{\gamma} ([0,x))$ and any boundary component other than the top.
Therefore, $Z_j$ intersects $Z_i$ which is a contradiction.
\end{enumerate}
 \end{proof}

The following lemma contains the main construction needed for the proof of Proposition~\ref{lem:surface}, and Lemmas~\ref{lem:fill-in} and~\ref{lem:fill-out} will let us confirm that it produces a $D+\varepsilon$-separator.

\begin{lemma}\label{lem:draw-z}
Let $\widetilde{\Sigma}$, $\widetilde{\gamma}$, and $Z$ be as in Proposition~\ref{lem:surface}.  Then for every $\varepsilon > 0$ there is a finite graph $Z'$ of $[0, \varepsilon] \times [0, L]$ with the following properties:
\begin{enumerate}
    \item $Z'$ contains the segment $\{\varepsilon\} \times [0, L]$.
    \item Let $Z'_1, \ldots, Z'_k$ be the path components of $Z'$, except for $\{\varepsilon\}\times [0, L]$ if it is its own component, and let $U'_1, \ldots, U'_\ell$ be the path components of the complement of $Z'$ in $[0, \varepsilon] \times [0, L]$.  Then every $Z'_i$ and every $U'_j$ contains a point of the vertical segment $\{0\} \times [0, L]$.  Abusing notation, let $\min(Z'_i)$ and $\max(Z'_i)$, or respectively $\inf(U'_j)$ and $\sup(U'_j)$, denote the infimal and supremal values of $t$ such that $(0, t) \in Z'_i$, or respectively $(0, t) \in U'_j$.
    \item 
There exist $Z_1, \ldots, Z_k$ in $\{Z_\alpha\}$ and $U_1, \ldots, U_\ell$ in $\{U_\beta\}$, such that for each $t \in [0, L]$, we have $\widetilde{\gamma}(t) \in Z_i$ if and only if $(0, t) \in Z'_i$, and $\widetilde{\gamma}(t) \in U_j$ if and only if $(0, t) \in U'_j$.

    \item If $(s, t) \in Z'_i$ (resp. $(s, t) \in U'_j$) and $t > \max(Z'_i)$ (resp. $t > \sup U'_j$), then $Z_i$ (resp. $U_j$) contains a point in a component of $\del \widetilde{\Sigma}$ other than the bottom component.  Similarly, if $(s, t) \in Z'_i$ (resp. $U'_j$) and $t < \min(Z'_i)$ (resp. $t < \inf(U'_j)$), then $Z_i$ (resp. $U_j$) contains a point in a component of $\del \widetilde{\Sigma}$ other than the top component.
\end{enumerate}
\end{lemma}

\begin{proof} We already know what the subsets $Z_1, \ldots, Z_k$ and $U_1, \ldots, U_\ell$ are: among the components of the intersection of $Z$ with the right side of $\widetilde{\Sigma}$, components $Z_1, \ldots, Z_k$ are those that touch $\widetilde{\gamma}$, and among the components of the complement of $Z$ in the right side of $\widetilde{\Sigma}$, components $U_1, \ldots, U_\ell$ are those that touch $\widetilde{\gamma}$.    
Our task is to construct a corresponding set $Z'_i \subseteq [0, \varepsilon] \times [0, L]$ for each $Z_i$.  Figure~\ref{fig:draw-z} and Figure~\ref{fig:draw-z1}  sketches what the construction may look like in different cases.

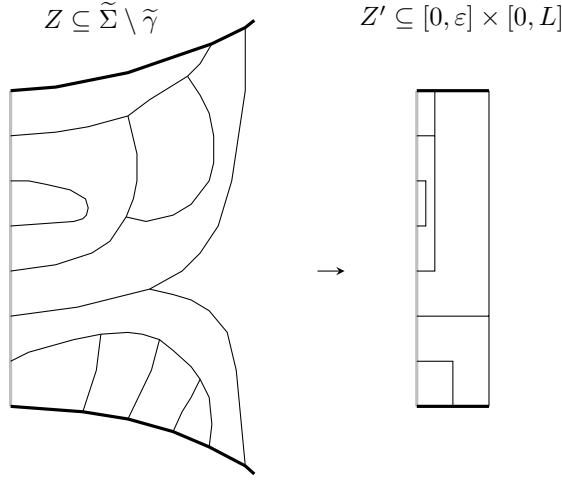
\begin{figure}
\begin{center}
\begin{tikzpicture}[>=stealth, scale=.12]
\draw[gray!50, very thick] (0, -15)--(0, 20);
\draw[very thick] (0, 20)--(5, 20.4)--(13, 22)--(22.3, 25.2)--(26, 27)--(27, 27.8);
\draw[very thick] (0, -15)--(8, -15.6)--(13, -16.4)--(18, -17.8)--(22, -19.5)--(26, -21.6)--(27, -22.5);

\draw (0, 15)--(5, 15.4)--(8.6, 16)--(13, 17.2)--(15.4, 19)--(19.6, 21.6)--(21, 23)--(22.3, 25.2);
\draw (0, 0)--(5, .7)--(9, 2)--(11, 3.3)--(12.8, 6)--(13.6, 8)--(14, 10)--(14, 13)--(13, 17.2);
\draw (12.8, 6)--(15, 5.5)--(18, 6.2)--(20, 7.4)--(22, 10)--(22.5, 12)--(22.5, 15)--(22, 17.5)--(21, 19.5)--(19.6, 21.6);
\draw (0, 5)--(7, 5.5)--(8, 5.8)--(8.4, 6.2)--(8.6, 7)--(8.3, 8)--(6.6, 9)--(4, 9.6)--(2, 10)--(0, 10);
\draw (0, -5)--(7.4, -4)--(15.4, -2)--(19, 0)--(21, 2)--(23, 5)--(24.5, 10)--(26, 20)--(26, 27);
\draw (15.4, -2)--(18, -2.6)--(20, -3.2)--(22, -4.4)--(24, -6.8)--(25, -11)--(26, -21.6);
\draw (0, -10)--(2, -9)--(7, -7.6)--(10, -7)--(13, -6.8)--(14.6, -7)--(16.5, -7.6)--(18, -8.8)--(20, -10.6)--(21, -12)--(22, -14)--(22.3, -17)--(22, -19.5);
\draw (10, -7)--(9.5, -11)--(8, -15.6);
\draw (16.5, -7.6)--(15.6, -11)--(13, -16.4);
\draw (21, -12)--(20, -14)--(19, -15.6)--(18, -17.8);
\node at (6, 30) {Part of $Z$ lying to the right of $ \widetilde{\gamma}$};
 \node at (50, 30) {$Z' \subseteq [0, \varepsilon] \times [0, L]$};
\draw[->] (34,0)--(37,0);

\begin{scope}[shift={(45, 0)}]
\draw[gray!50, very thick] (0, -15)--(0, 20);
\draw[very thick] (0, 20)--(8, 20);
\draw[very thick] (0, -15)--(8, -15);
\draw (0, -5)--(8, -5) (8, -15)--(8, 20);
\draw (0, -10)--(4, -10)--(4, -15);
\draw (0, 0)--(2, 0)--(2, 20) (0, 15)--(2, 15);
\draw (0, 5)--(1, 5)--(1, 10)--(0, 10);
\end{scope}
\end{tikzpicture}
\end{center}
\caption{For each $Z_i$, the corresponding $Z'_i$ connects to the same points along the left edge.  If $Z_i$ connects to $\del \widetilde{\Sigma}$, then $Z'_i$ connects to the top or bottom of the rectangle.}\label{fig:draw-z}
\end{figure}

First we consider the special case where some $Z_i$ contains a point from the top component of $\del \widetilde{\Sigma}$ and also a point from the bottom component of $\del \widetilde{\Sigma}$ (see Figure~\ref{fig:draw-z}).  We may reorder to call this set $Z_1$.  Then we construct $Z'_1$ to consist of the vertical segment $\{\varepsilon\} \times [0, L]$, along with horizontal segments $[0, \varepsilon] \times \{t\}$ for every $t$ with $\widetilde{\gamma}(t) \in Z_1$.

Next, we construct the remaining sets $Z'_i$ iteratively, in a specific order.  Suppose that we have selected which sets are labeled $Z_1, \ldots, Z_{i-1}$ in the ordering, and constructed the corresponding sets $Z'_1, \ldots, Z'_{i-1}$.  To select $Z_i$, look at the points of $\widetilde{\gamma} \cap Z$ in sequence from bottom to top, and select a point among these that is not in $Z_1, \ldots, Z_{i-1}$ but is adjacent, in the sequence, to one of those points.  Then let $Z_i$ be the component of $Z$ containing this point.

The set $Z'_i$ will consist of a vertical segment at horizontal coordinate $\frac{\varepsilon}{2^i}$, along with horizontal segments $[0, \frac{\varepsilon}{2^i}] \times \{t\}$ for every $t$ with $\widetilde{\gamma}(t) \in Z_i$.  The extent of the vertical segment depends on $Z_i$ in the following way.  If $Z_i$ contains a point from the top component of $\del \widetilde{\Sigma}$, then the vertical segment is $\{\frac{\varepsilon}{2^i}\}\times[\min(Z'_i), L]$.  If $Z_i$ contains a point from the bottom component of $\del \widetilde{\Sigma}$, then the vertical segment is $\{\frac{\varepsilon}{2^i}\}\times[0, \max(Z'_i)]$.  Because of the special case we are in, $Z_1$ prevents $Z_i$ from containing points from any other boundary components. 
If $Z_i$ does not contain a point of $\del \widetilde{\Sigma}$, then the vertical segment is $\{\frac{\varepsilon}{2^i}\}\times[\min(Z'_i), \max(Z'_i)]$.

After doing this process for each $i$, the resulting set $Z' = \bigcup_{i=1}^k Z'_i$ should look very much like the right side of $Z$, but rectilinear.  Our choice of how to order the sets $Z_i$, together with our choice of vertical segments moving toward the left, guarantees that none of the sets $Z'_i$ intersect each other.

Still in the special case, let us check the properties specified by the lemma statement. 
We find that properties (1), (2) and (3) are automatic for $Z'_i$ from the construction.  Property (4) holds for $Z_i$ because for each $Z'_i$ the vertical segment reaches the top of the rectangle exactly when $Z_i$ touches the top component of $\del \widetilde{\Sigma}$, and the vertical segment reaches the bottom of the rectangle exactly when $Z_i$ touches the bottom component of $\del \widetilde{\Sigma}$.

It remains to check the properties for $U'_j$. For each  $U_j'$, we let $U_j$ to be a path component of the complement of $Z$ on the right side of $\widetilde{\Sigma}\setminus \widetilde{\gamma}$ that has a common point along $\widetilde{\gamma}$ with $U'_j$: there exists a $t$ such that $(0,t)\in U_j'$ and $\widetilde{\gamma}(t)\in U_j$.
Property (3) and (4) are the only nontrivial properties to check.

To verify (3),
 for simplicity, we simplify notation by identifying $\{0\} \times [0,L]$ with $\widetilde{\gamma}$ itself.  
 We aim to show that $U_j'\cap\widetilde{\gamma}=U_j\cap \widetilde{\gamma}$.
Suppose $y\in U_j'\cap \widetilde{\gamma}$. 
 Let $x\in \widetilde{\gamma}$ be a point lying in both $U_j'$ and $U_j$.
 Let $I$ be the interval between $x$ and $y$ along $\widetilde{\gamma}$.
 Since $x$ and $y$ belong to the same path component of $[0,\varepsilon]\times [0,L]\setminus Z'$, no component $Z_i'$ connects the interval $I$ with a point in the boundary of $[0,\varepsilon]\times [0,L]$ that lies outside $I$.
 By the construction of the set $Z_i'$, it follows that there is no $Z_i$ that connects $I$ to a point in $\partial{\Sigma}\cup \widetilde{\gamma}\setminus I$ on the (closure of the) right side of $\widetilde{\Sigma}\setminus \widetilde{\gamma}$.
 By Lemma~\ref{l:Uj}(1), we have $y\in U_j$.
 This proves $U_j'\cap\widetilde{\gamma}\subset U_j\cap \widetilde{\gamma}$. 
 The reverse inclusion follows by a similar argument, completing the proof of property (3).

Next we check the property (4) for $U_j'$.
Suppose $(s,t)\in U_j'$ and $t>\sup(U_j')$.  
Pick $x< \sup (U_j')$ such that $\{0\}\times [x,\sup (U_j'))\subset U_j'$.
There exists some $Z_i'$ that contains $(0,\sup(U_j'))$.
Moreover, by an argument similar to Lemma~\ref{l:Uj}(2),
 no arc of $Z'_i$ connects $(0,\sup (U_j'))$ to a point in $\{0\}\times [0,\sup (U_j'))$ or $[0,\varepsilon]\times \{0\}$ or $\{\varepsilon\} \times [0,L]$: because each such arc separates $(0,x)$ from $(s,t)$ by the construction of $Z_i'$.
It follows that $Z_i$ contains $\widetilde{\gamma}(\sup(U_j))$ and there is no arc of $Z_i$ that connects $\widetilde{\gamma}(\sup (U_j))$ to $\widetilde{\gamma}([0,\sup (U_j)))$ or the bottom component or any other boundary component on the right side of $\widetilde{\Sigma}\setminus \widetilde{\gamma}$ other than the top boundary component.
Therefore $U_j$ intersects the top component of $\del \widetilde{\Sigma}$ by Lemma~\ref{l:Uj}(2).
Similarly, we can prove the analogous statement when $(s,t)\in U_j'$ and $t<\inf(U_j')$.

If we are not in the special case, then there exists some component $U_j$ that either contains a point in a third component of $\del \widetilde{\Sigma}$, or contains a point in the top component of $\del \widetilde{\Sigma}$ and also a point in the bottom component of $\del \widetilde{\Sigma}$ (see Figure~\ref{fig:draw-z1}).  Let $\widetilde{\gamma}(p)$ be a point of $\widetilde{\gamma} \cap U_j$.  Note that for each $Z_i$, the points where it meets $\widetilde{\gamma}$ are either all above or all below $\widetilde{\gamma}(p)$, and those $Z_i$ that are above do not touch the bottom component of $\del \widetilde{\Sigma}$, while those $Z_i$ that are below do not touch the top component of $\del \widetilde{\Sigma}$.

\begin{figure}[htp]
\begin{center}
\begin{tikzpicture}[>=stealth, scale=.12]
\draw[gray!50, very thick] (0, -15)--(0, 20);
\draw[very thick] (0, 20)--(5, 20.4)--(13, 22)--(22.3, 25.2)--(26, 27)--(27, 27.8);
\draw[very thick] (0, -15)--(8, -15.6)--(13, -16.4)--(18, -17.8)--(22, -19.5)--(26, -21.6)--(27, -22.5);

\draw (0, 15)--(5, 15.4)--(8.6, 16)--(13, 17.2)--(15.4, 19)--(19.6, 21.6)--(21, 23)--(22.3, 25.2);
\draw (0, 0)--(5, .7)--(9, 2)--(11, 3.3)--(12.8, 6)--(13.6, 8)--(14, 10)--(14, 13)--(13, 17.2);
\draw (12.8, 6)--(15, 5.5)--(18, 6.2)--(20, 7.4)--(22, 10)--(22.5, 12)--(22.5, 15)--(22, 17.5)--(21, 19.5)--(19.6, 21.6);
\draw (0, 5)--(7, 5.5)--(8, 5.8)--(8.4, 6.2)--(8.6, 7)--(8.3, 8)--(6.6, 9)--(4, 9.6)--(2, 10)--(0, 10);
\draw (0, -5)--(7.4, -4)--(15.4, -2)--(19, 0)--(21, 2)--(23, 5)--(24.5, 10)--(26, 20)--(26, 27);

\draw (0, -10)--(2, -9)--(7, -7.6)--(10, -7)--(13, -6.8)--(14.6, -7)--(16.5, -7.6)--(18, -8.8)--(20, -10.6)--(21, -12)--(22, -14)--(22.3, -17)--(22, -19.5);
\draw (10, -7)--(9.5, -11)--(8, -15.6);
\draw (16.5, -7.6)--(15.6, -11)--(13, -16.4);
\draw (21, -12)--(20, -14)--(19, -15.6)--(18, -17.8);

\begin{scope}[shift={(40, 0)}]
\draw[gray!50, very thick] (0, -15)--(0, 20);
\draw[very thick] (0, 20)--(5, 20.4)--(13, 22)--(22.3, 25.2)--(26, 27)--(27, 27.8);
\draw[very thick] (0, -15)--(8, -15.6)--(13, -16.4)--(18, -17.8)--(22, -19.5)--(26, -21.6)--(27, -22.5);

\draw (0, 15)--(5, 15.4)--(8.6, 16)--(13, 17.2)--(15.4, 19)--(19.6, 21.6)--(21, 23)--(22.3, 25.2);
\draw (0, 0)--(5, .7)--(9, 2)--(11, 3.3)--(12.8, 6)--(13.6, 8)--(14, 10)--(14, 13)--(13, 17.2);
\draw (12.8, 6)--(15, 5.5)--(18, 6.2)--(20, 7.4)--(22, 10)--(22.5, 12)--(22.5, 15)--(22, 17.5)--(21, 19.5)--(19.6, 21.6);
\draw (0, 5)--(7, 5.5)--(8, 5.8)--(8.4, 6.2)--(8.6, 7)--(8.3, 8)--(6.6, 9)--(4, 9.6)--(2, 10)--(0, 10);
\draw (0, -5)--(7.4, -4)--(15.4,-2)--(19, 0);

\draw[very thick] (26,10)..controls (17,0)..(26,-11);
\draw (0, -10)--(2, -9)--(7, -7.6)--(10, -7)--(13, -6.8)--(14.6, -7)--(16.5, -7.6)--(18, -8.8)--(20, -10.6)--(21, -12)--(22, -14)--(22.3, -17)--(22, -19.5);
\draw (10, -7)--(9.5, -11)--(8, -15.6);
\draw (16.5, -7.6)--(15.6, -11)--(13, -16.4);
\draw (21, -12)--(20, -14)--(19, -15.6)--(18, -17.8);
\filldraw[black] (0,-7.5) circle (8pt) node[anchor=east]{$\widetilde{\gamma}(p)$};
\end{scope}

\filldraw[black] (0,-7.5) circle (8pt) node[anchor=east]{$\widetilde{\gamma}(p)$};
\draw[->] (75,0)--(78,0);

\begin{scope}[shift={(85, 0)}]
\draw[gray!50, very thick] (0, -15)--(0, 20);
\draw[very thick] (0, 20)--(12, 20);
\draw[very thick] (0, -15)--(12, -15);
\draw (0, -5)--(4, -5)--(4,20); 
\draw (12, -15)--(12, 20);
\draw (0, -10)--(7, -10)--(7, -15);
\draw (0, 0)--(2, 0)--(2, 20) (0, 15)--(2, 15);
\draw (0, 5)--(1, 5)--(1, 10)--(0, 10);

\end{scope}
\end{tikzpicture}
\end{center}
\caption{On the left are two typical instances of $Z$ in the non--special case; on the right is the corresponding $Z'$.}\label{fig:draw-z1}
\end{figure}
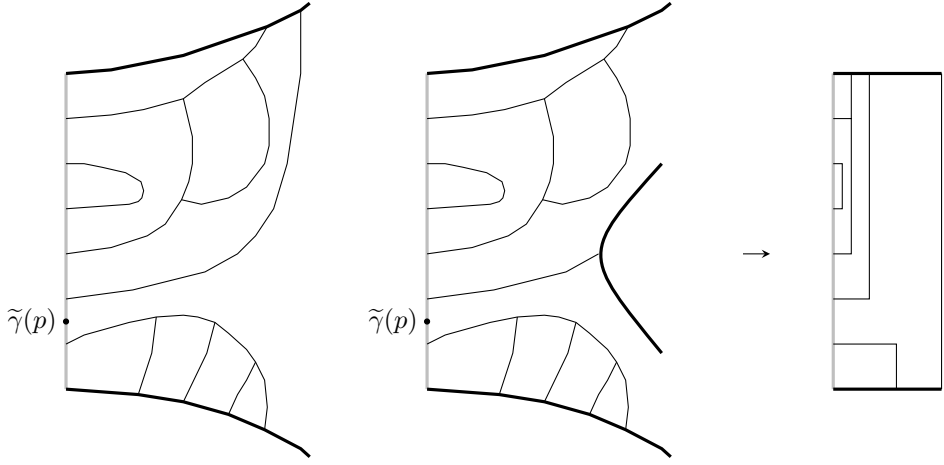

We start by letting $\{\varepsilon\} \times [0, L]$ be its own component of $Z'$.  Then we continue as in the special case, with the following modifications.  As $Z_1$ and $Z_2$ we select the components of the two points of $\widetilde{\gamma} \cap Z$ on either side of $\widetilde{\gamma}(p)$. 
The rest of the process for ordering $Z_3, \ldots, Z_k$ is the same as in the special case.  When determining the extent of the vertical segment in $Z'_i$, we use the following rule.  If $Z_i$ contains a point of $\del \widetilde{\Sigma}$ and $\widetilde{\gamma}\cap Z_i$ is above $\widetilde{\gamma}(p)$, then the vertical segment is $\{\frac{\varepsilon}{2^i}\} \times [\min(Z'_i), L]$.  If $Z_i$ contains a point of $\del \widetilde{\Sigma}$ and $\widetilde{\gamma} \cap Z_i$ is below $\widetilde{\gamma}(p)$, then the vertical segment is $\{\frac{\varepsilon}{2^i}\}\times [0, \max(Z'_i)]$.  If $Z_i$ does not contain a point of $\del \widetilde{\Sigma}$, then the vertical segment is $\{\frac{\varepsilon}{2^i}\}\times[\min(Z'_i), \max(Z'_i)]$.  As in the special case, each $Z'_i$ is a rectilinear version of $Z_i$, but we use the point $p$ to determine which paths to third boundary components of $\del \widetilde{\Sigma}$ become paths to the top of $[0, \varepsilon] \times [0, L]$ versus the bottom.

As in the special case, we only need to check properties (3) and (4), and our construction guarantees these properties for each $Z'_i$.  
We choose $U_j$ in relation to $U_j'$ in the same way as in the special case.
Property (3) for \( U_j' \) is a direct consequence of the construction of $U_j'$ and Lemma~\ref{l:Uj}(1), with the argument proceeding exactly as in the special case.
It remains to prove property (4) for $U_j'$. For each $U'_j$, if it is the component adjoining the right edge $\{\varepsilon\}\times[0, L]$, then the property holds because of how we have chosen $p$.  
For any other $U'_j$, the proof is the same as in the special case.
\end{proof}

Finally we are ready to finish the proof of Proposition~\ref{lem:surface}, completing the proof of Theorem~\ref{thm:surface} for the case of surface with boundary.

\begin{proof}[Proof of Proposition~\ref{lem:surface}]
First we note that $D \geq L$.  This is because either a path component of $Z$ that crosses $\widetilde{\gamma}$ contains points from two different components of $\del \widetilde{\Sigma}$, or there is a point in $\widetilde{\gamma}$ that has paths disjoint from $Z$ to two different components of $\del \widetilde{\Sigma}$. 
Either way, a component of $Z$ or its complement has diameter at least $L$.

We construct the part of $Z'$ in $[-M, -M+\varepsilon]$ by applying Lemma~\ref{lem:draw-z} to $\varepsilon$, and we construct the part of $Z'$ in $[M-\varepsilon, M]$ by applying Lemma~\ref{lem:draw-z} to $\varepsilon$, but in mirror image, with the left side of $\widetilde{\Sigma} \setminus \widetilde{\gamma}$ playing the role of the right side.  Then we add vertical segments $\{s\} \times [0, L]$ along the strip, including the values $s = -M + \varepsilon$ and $s = M-\varepsilon$ and with consecutive values spaced less than $\varepsilon$ apart.

To check that the resulting set $Z'$ is a $D+\varepsilon$-separator, first we check that for every point $x$ in the left side of $\widetilde{\Sigma} \setminus \widetilde{\gamma}$, if a connected component of $Z'$ or its complement contains both $x$ and a point $(s, t)$ of the strip, then the distance from $x$ to $(s, t)$ is at most $D + \varepsilon$.
To do this, it suffices to show that the distance from $x$ to $(-M, t)$ is at most $D$.  Let $Z'_i$ or $U'_j$ be the relevant component of $Z'$ or its complement.  If $t$ is in $[\min(Z'_i), \max(Z'_i)]$ or $[\inf(U'_j), \sup(U'_j)]$, then Lemma~\ref{lem:fill-in} implies that $(-M, t)$ is in $B(x, D)$.  Otherwise, property (4) of Lemma~\ref{lem:draw-z} allows us to find a boundary point $b$ in $Z'_i$ or $U'_j$, and then applying Lemma~\ref{lem:fill-out} we have that $(-M, t)$ is in $B(x, D)$.

The case of the distance between a point in the right side of $\widetilde{\Sigma} \setminus \widetilde{\gamma}$ and a point in the strip is exactly analogous.  If two points are on the left side, or two points are on the right side, their distance does not change, and our construction guarantees that if they are in the same component of $Z'$ or its complement, then they are also in the same component of $Z$ or its complement.  
If two points are in the strip, and are in the same component of $Z'$ or its complement, then their distance is at most $L + \varepsilon$, which is at most $D + \varepsilon$.  Thus, every pair of points in the same component of $Z'$ or its complement have distance at most $D + \varepsilon$.
\end{proof}

Now we deal with the remaining cases of Theorem~\ref{thm:surface}. If $\Sigma$ is a sphere, then $\widetilde{\Sigma} = \Sigma$ and the theorem is a tautology.  If $\Sigma$ has higher genus but has no boundary, then $\UW_1(\widetilde{\Sigma}) = \infty$ and the conclusion of the theorem is vacuously true. The only nontrivial case is when $\Sigma$ is a $\RP2$. The proof in this case is similar to the surface with boundary case. We need the following analogue of Proposition~\ref{lem:surface}.

\begin{proposition}[Main proposition for the $\RP2$ case]\label{lem:RP2}
Suppose $\Sigma$ is a  $\mathbb{R}\mathbb{P}^2$.
Let $\gamma$ be a homotopically nontrivial loop with the shortest length in $\Sigma$.
    Let $\widetilde{\gamma}$ be the lift of $\gamma$ in $\widetilde{\Sigma}$. 
  For any $M > 0$, let $\widetilde{\Sigma}'$ be the result of cutting apart $\widetilde{\Sigma}$ along $\widetilde{\gamma}$ and gluing in a band $[-M, M] \times \widetilde{\gamma}$, such that its ends $\{-M\} \times \widetilde{\gamma}$ and $\{M\} \times \widetilde{\gamma}$ attach isometrically to the two cut copies of $\widetilde{\gamma}$. Let $\Sigma'$ be the quotient of $\widetilde{\Sigma}'$ by the canonical antipodal map induced from the antipodal map on $\widetilde{\Sigma}$.  Suppose that $Z$ is a $D$-separator in $\widetilde{\Sigma}$ for some $D > 0$, containing only finitely many points of $\widetilde{\gamma}$.  Then for any $\varepsilon>0$ there exists an $M$ such that the corresponding $\Sigma'$ admits a $D+\varepsilon$-separator.
\end{proposition}
We first prove Theorem~\ref{thm:surface} for the $\RP2$ case assuming the above lemma.

\begin{proof}[Proof of Theorem~\ref{thm:surface} for $\RP2$]Suppose $\Sigma$ is $\mathbb{R}\mathbb{P}^2$ and let $\gamma$ be a homotopically nontrivial loop in $\Sigma$ of shortest length. 
We know that $\widetilde{\Sigma}$ admits a $D$-separator for some $D$ close to $\UW_1(\widetilde{\Sigma})$.  We apply proposition~\ref{lem:RP2} to the lifts of $\widetilde{\gamma}$ to obtain a $D+\varepsilon$-separator on the quotient $\Sigma'$ for some $M$.
There is a map $\Sigma'\to \Sigma$ that sends the quotient of the band $[-M,M]\times \widetilde{\gamma}$ in $\Sigma'$ to a small tubular neighborhood of $\gamma$ in $\Sigma$.
This map may be taken to be a  homeomorphism $\Sigma'\to \Sigma$ that increases distances by at most $\varepsilon$.
Thus, we obtain a $D+2\varepsilon$-separator on $\Sigma$. Since $\varepsilon$ may be taken arbitrarily small, we obtain $\UW_1(\Sigma)\leq D$.
\end{proof}

It remains to prove proposition~\ref{lem:RP2}. For that we need the following lemma which can be viewed as an analogue of Lemma~\ref{lem:fill-in}.
\begin{lemma}\label{lem:fill-in1}
Let $\Sigma$ and $\gamma$ be as in proposition~\ref{lem:RP2}. Then the following holds.
\begin{enumerate}
    \item For any two points on the lift $\widetilde{\gamma}$, there exists a geodesic segment connecting them that lies entirely within $\widetilde{\gamma}$.
\item Let $x$ be a point in the left component of $\widetilde{\Sigma} \setminus \widetilde{\gamma}$, and suppose that the ball $B(x, D)$ contains two points $a_1$ and $a_2$ on $\widetilde{\gamma}$, as well as a path from $a_1$ to $a_2$ in the right component of $\widetilde{\Sigma} \setminus \widetilde{\gamma}$.
 Let $q:\widetilde{\Sigma}\to \Sigma$ be the covering projection and $[a_1,a_2]$ be a minimizing geodesic path in $\widetilde{\gamma}$ between $a_1$ and $a_2$. 
 Then $B(q(x), D)$ contains $q([a_1,a_2])$.    
\end{enumerate}
\end{lemma}
\begin{proof}
\begin{enumerate}
\item Suppose the length of $\gamma$ is $L$.
We first claim that the distance between any two antipodal points in $\widetilde{\gamma}$ is $L$.
If not, then we can take such a pair of antipodal points which are $<L$ distance apart and hence there is a geodesic between these two points with length $<L$.
Projecting this geodesic to $\Sigma$, we obtain a nontrivial loop with length $<L$, which is a contradiction.
It follows from the same argument that both of the paths between any two antipodal points in $\widetilde{\gamma}$ are geodesic.
 If  $a,b\in\widetilde{\gamma}$, then $b$ lies on a geodesic joining $a$ and its antipode on $\widetilde{\gamma}$.
Therefore, there is a geodesic between $a$ and $b$ in $\widetilde{\gamma}$.

\item Note that the two connected components of $\widetilde{\gamma}\setminus \{a_1,a_2\}$ give two paths between $a_1$ and $a_2$, and one of them is a minimizing geodesic path $[a_1,a_2]$.
Moreover, $q([a_1,a_2])$ is contained in the $q$-image of the other path. 
Since the map $q:\widetilde{\Sigma}\to\Sigma$ is non-increasing, it is enough to prove that $B(x,D)$ contains at least one path between $a_1$ and $a_2$ on $\widetilde{\gamma}$.

Let $\pi$ be the path in $B(x,D)$ from $a_1$ to $a_2$ in the right component of $\widetilde{\Sigma}\setminus \widetilde{\gamma}$.  We draw geodesics back to $x$ from every point of $\pi$. 
Suppose on the contrary $B(x,D)$ contains neither path between $a_1$ and $a_2$ on $\widetilde{\gamma}$.
Then there are at least two points $a,a'$ on $\widetilde{\gamma}$ that are outside $B(x,D)$. Moreover, these two points cut $\widetilde{\gamma}$ into two connected components $C_1$ and $C_2$ such that  $a_1\in C_1$ and $a_2\in C_2$.
Arguing as in Lemma~\ref{lem:fill-in}, we can obtain point $b$ in $\pi$ such that there are two geodesics from $x$ to $b$, one crossing $\widetilde{\gamma}$ in $C_1$ at a point $c_1$, and the other crossing $\widetilde{\gamma}$ in $C_2$ at a point $c_2$.
Either $a$ or $a'$ must lie on the minimizing geodesic on $\widetilde{\gamma}$ between $c_1$ and $c_2$.
Without loss of generality, suppose $a$ lies on the minimizing geodesic on $\widetilde{\gamma}$ between $c_1$ and $c_2$.
We have
\begin{align*}
2D \leq 2\cdot d(x, a)  &\leq d(x, c_1) + d(c_1, a) + d(x, c_2) + d(c_2, a) \\
&= d(x, c_1) + d(c_1, c_2) + d(x, c_2) \\
&\leq d(x, c_1) + d(c_1, b) + d(x, c_2) + d(c_2, b) \leq 2 \cdot d(x, b) < 2D,
\end{align*}
giving a contradiction.

\end{enumerate}
\end{proof}

We now proceed to prove proposition~\ref{lem:RP2} which will finish the proof of Theorem~\ref{thm:surface}.  

\begin{proof}[Proof of proposition~\ref{lem:RP2}] Let $Z$ be a $D$-separator on $\widetilde{\Sigma}$.
    Let $\gamma$ be a homotopically nontrivial loop in $\Sigma$ of shortest length and 
    let $\widetilde{\gamma}$ be the lift of $\gamma$ in $\widetilde{\Sigma}$.
    
    We first note that $\diam(\widetilde{\gamma})\leq D$.
    Since $\UW_1(\widetilde{\Sigma})\leq D$ and $\widetilde{\Sigma}$ is simply connected, there exists a map from $\widetilde{\Sigma}$ to a compact tree such that the diameter of each fiber is arbitrarily close to $D$. Since we can embed a compact tree into $\mathbb{R}^2$, we have a map $\widetilde{\Sigma}\to \mathbb{R}^2$ such that diameter of each fiber is arbitrarily close to $D$.
    By Borsuk--Ulam theorem, there exists at least one fiber that contains two antipodal points of $\widetilde{\Sigma}$, where antipodal points are those which get identified in $\Sigma$.
    Consequently, there exist two antipodal points in $\widetilde{\Sigma}$ with distance arbitrarily close to $D$.
    Taking the geodesic joining these two antipodal points and then projecting to $\Sigma$ gives a nontrivial loop on $\Sigma$ with length arbitrarily close to $D$.
    Since $\gamma$ was chosen to be length minimizing, we obtain that $\diam(\widetilde{\gamma})\leq D$.
   
For convenience, we will assume that each connected component of $Z$ is a simple loop by replacing each component by the boundary of its thin regular neighborhood. 
Also, without loss of generality, we can assume that $Z$ intersects $\widetilde{\gamma}$ transversely, and in particular at finitely many points.
Consequently, we can assume that each connected component of $Z$ that intersects $\widetilde{\gamma}$ does so in at least two points.

Note that $\widetilde{\Sigma}\setminus \widetilde{\gamma}$ is a disjoint union of two disks which we refer to as the left disk and the right disk.  
Next we are going to iteratively replace each arc $l_i$ of $Z$  connecting two points in $\widetilde{\gamma}$ in the right disk by replicating them on the part of $[-M,M]\times \widetilde{\gamma}$ next to the left disk.
Suppose that we have selected which arcs are labeled $l_1,l_2,\ldots,l_{i-1}$ in the ordering, and constructed the corresponding modified arcs $l'_1,l'_2,\ldots, l'_{i-1}$.
We let $l_i$ be the arc with the maximal distance between its end points among the remaining arcs. 
Let $a$ and $b$ be the end points of $l_i$ and let $[a,b]$ be a geodesic between $a$ and $b$ on $\widetilde{\gamma}$.
Note that such geodesic exists by Lemma~\ref{lem:fill-in1}(1).

We replace $l_i$ with a vertical circular segment $\{-M+\frac{\varepsilon}{2^i}\}\times[a, b]$ along with horizontal segments $[-M,-M+\frac{\varepsilon}{2^i}]\times \{a,b\}$ and call it $l'_i$.  
Next, we add vertical circles $\{s\}\times \widetilde{\gamma}$ along the band $[-M,0]\times \widetilde{\gamma}$, including the values $s=-M+\varepsilon$ and $s=0$ and with consecutive values spaced less than $\varepsilon$ apart.  
This gives us the $Z'$ on the left half of the $\widetilde{\Sigma}'$.
By applying the antipodal map we get the $Z'$ on the other half of $\widetilde{\Sigma}'$.

 For each connected component $Z_i$ of $Z$ on the left disk of $\widetilde{\Sigma}$, let $Z_i'$ be the modified $Z_i$. 
We now check that $Z_i'$'s do not intersect each other by construction.
It is enough to show that for any $i>j$, $l'_{i}$ does not intersect $l'_j$.
On the contrary, suppose $l_i'$ intersects $l_j'$ for some $i>j$.
We let $\gamma_i$ denote a length-minimizing geodesic along $\widetilde{\gamma}$ joining the endpoints of $l_i$.
Then at least one of the end points of $l_j$ is on $\gamma_i$.
Since $j<i$, by the choice of $l_i$, the length of $\gamma_i$ is at most the length of $\gamma_j$ and therefore at least one endpoint of $l_j$ is outside $\gamma_i$.
Since we are in a disk, these together imply that $l_i\cap l_j\neq \emptyset$ which is a contradiction.

Next, we check that the projection of the set $Z'$ under the covering projection map $q':\widetilde{\Sigma}'\to \Sigma'$ gives a $D+\varepsilon$-separator of $\Sigma'$ for small enough $M$.

If two points are on the left disk, then their distance does not increase in the quotient and our construction guarantees that if they are in the same component of $Z'$ or its complement, then they are also in the same component of $Z$ or its complement.  If two points are in the band, and are in the same component of $Z'$ or its complement, then their distance is at most $\diam(\widetilde{\gamma}) + \varepsilon$, which is at most $D + \varepsilon$.  Thus, the images under the quotient of every pair of such points in the same component of $Z'$ or its complement have distance at most $D + \varepsilon$. 

It is now enough to show that we can choose an $M$ so small that for every point $x$ in the left disk of $\widetilde{\Sigma} \setminus \widetilde{\gamma}$, if a connected component of $Z'$ or its complement contains both $x$ and a point $(s, t)$ of the band $[-M,M]\times \widetilde{\gamma}$, then the distance from $q'(x)$ to $q'(s, t)$ is at most $D + \varepsilon$.    
To do this, it suffices to show that for some $M$, the distance from $q'(x)$ to $q'(-M, t)$ is at most $D+\varepsilon$.

To this end, consider the horizontal path $\beta$ from $(s,t)$ to $(-M,t)$.
Suppose, $\beta$ does not intersect $Z'$.
Then $(-M,t)$ and $x$ both belong to the same component of $Z'$ or its complement.
By our construction of $Z'$, it follows that $(-M,t)$ and $x$ belong to the same component of $Z$ or its complement and hence their distance is at most $D$.
Consequently, the distance from $q'(x)$ to $q'(-M, t)$ is at most $D+\varepsilon$ in this case.
Now suppose $\beta$ does intersect some $l'_i$ of $Z_i'$ at $(s_i,t)$.
Without loss of generality, we assume that no other $l_j$ intersects $\beta$ between $(s,t)$ and $(s_i,t)$.
It follows that either $x\in Z_i$ or $x$ is contained in some $U_j$ such that the closure of $U_j$ contains $Z_i$.
In either cases, $B(x,D)$ contains $l_i$.
Since $\beta$ intersects $l_i'$, it follows that $t$ lives in the minimizing geodesic connecting the end points of $l_i$.
Lemma~\ref{lem:fill-in1} implies $d(q(x),q(t))\leq D$ where $q:\widetilde{\Sigma}\to \Sigma$ is the covering projection.
There exists a map $\Sigma'\to \Sigma$ that  sends the quotient of the band $[-M,M]\times \widetilde{\gamma}$ in $\Sigma'$ to $\gamma$ in $\Sigma$ by forgetting the first coordinates and maps $q'(x)$ to  $q(x)$ for all $x$ outside the band.
This map changes the distance on the order of $M$, and therefore we can choose small enough $M$ depending on $\varepsilon$ to have the following
\begin{align*}
    d(q'(x),q'(-M,t))\leq d(q(x),q(t))+\varepsilon \leq D+\varepsilon.
\end{align*}
This finishes the proof.

\end{proof}

\begin{remark}
    It is reasonable to conjecture that Theorem~\ref{thm:surface} holds for any Riemannian manifold. The simplest case that we do not know the answer to is the three-dimensional handlebody.
    \begin{question}
        Does there exist a constant $c>0$, such that for any 3-dimensional handlebody $X$ with a Riemannian metric, we have $\UW_1(X)\leq c\cdot \UW_1(\widetilde{X})$? 
    \end{question}
\noindent    The question remains open even in the case where $X$ is a handlebody of genus 2.
\end{remark}

\begin{remark}
    Balitskiy and Berdnikov~\cite{BB21} proved that if a closed Riemannian manifold $M$ has first $\mathbb{Z}_2$-Betti number $\beta$ and every unit ball in $M$ has 1-width less than $\frac{1}{15}$, then the Uryson 1-width of $M$ satisfies $\UW_1(M) \leq \beta + 1$.
In light of Theorem~\ref{thm:surface}, one might wonder whether the dependence on the first Betti number can be removed in the case of surfaces with boundary. However, this is not the case.
For example, consider the surface $M$ from Example~\ref{e:example}, and remove a unit-radius ball to obtain a surface with boundary. This new surface still has large 1-width---on the order of $R$.
On the other hand, the covering map $p : \widehat{M} \to M$ is an isometry when restricted to balls of radius 1, provided $R \gg 1$.
Since $\UW_1(\widehat{M})$ is small (on the order of 1), it follows that the 1-width of unit balls in $M$ is also of order 1. This illustrates that local control of 1-width (at the scale of unit balls) is not sufficient to bound the global 1-width even for surfaces with boundaries.
\end{remark}

\section{Reduction to low dimensions}\label{s4}
In this section, our goal is to prove Theorem~\ref{thm:low dim reduction}.
The first part of Theorem~\ref{thm:low dim reduction} is established in Proposition~\ref{p:2dim}, which is then used to prove the second part, appearing as Theorem~\ref{t:4dim}.
We start with the following.
\begin{proposition}[Manifold reduction]\label{p:mfld}
Suppose that $\{X_n\}_{n=1}^\infty$ is a sequence of compact Riemannian polyhedra such that $\{\dim(X_n)\}$ is bounded and the ratio $\displaystyle\frac{\UW_1(X_n)}{\UW_1(\widetilde{X_n})}$ is unbounded. Then there is a related sequence $\{Z_n\}_{n=1}^\infty$ of closed Riemannian manifolds such that $\dim(Z_n)=2\dim(X_n)$ and the ratio $\displaystyle\frac{\UW_1(Z_n)}{\UW_1(\widetilde{Z_n})}$ is unbounded.
\end{proposition}

To prove the proposition, we need the next two lemmas. The first lemma shows that for a given Riemannian polyhedron, there is  a closed Riemannian manifold which is very close to the Riemannian polyhedron in terms of both the geometry and topology.

\begin{lemma}\label{l:complex to mfld}
    Let $(K,d)$ be an $n$-dimensional compact Riemannian polyhedron and $n\geq 2$. Then for each $i\in \mathbb{N}$, there exists a closed Riemannian  $2n$-manifold $(M,d_i)$ and a surjective map $p:M\rightarrow K$ such that the following holds.
    \begin{enumerate}
\item If $A\subset K$ is connected, then $p^{-1}(A)$ is connected.    
        \item Every open cover of $K$ has a refinement $\{U_\alpha\}$ such that $p^{-1}(U_\alpha)$ is simply connected for each $\alpha$.
\item $|d(p(x),p(y))- d_i(x,y)|\leq \frac{1}{i}$ for all $x,y\in M$.         
        \item $p$ is $\pi_1$-injective.
    \end{enumerate}
\end{lemma}

\begin{proof}
    Choose a PL embedding of $K$ into $\mathbb{R}^{2n+1}$ and then take a regular neighborhood $N$ of $K$ in $\mathbb{R}^{2n+1}$. 
    Such neighborhood exists and moreover there is a projection map $N\rightarrow K$ which is a homotopy equivalence and  whose restriction $p:\partial N\rightarrow K$ satisfies the following property:
    If $x$ is a point in the interior of a $k$-simplex in $K$, then $p^{-1}(x)$ has a homotopy type of a $(2n-k)$-sphere with some $(n-k-1)$-complex removed where this $(n-k-1)$-complex is the link of $\sigma^k$ in $K$.    
    Since $n\geq 2$, the removed complex has codimension at least $3$ in the sphere, so these preimages are simply connected. We let $M:=\partial N$.

 To prove property (1), we observe that $p$ is a closed map because $M$ is compact. Since $p$ is a surjective, closed map and each fiber of $p$ is connected, it follows that inverse image of any connected subset under $p$ is also connected.
 
    To prove property (2), we first observe that for each $x\in K$, there exists a small enough neighborhood $U$ of $x$ such that $p^{-1}(U)$ has a homotopy type of $p^{-1}(x)$. 
    The claim follows because $p^{-1}(x)$ is simply connected.

   Since $K$ is a geodesic space, we can invoke the main result of~\cite{FO} to endow $\partial N$ with a Riemannian metric with the property (3).
   Let $M$ denote the manifold $\partial N$ endowed with this Riemannian metric.
   
 To prove property (4), suppose $f:S^1\rightarrow M$ is a loop such that $p\circ f$ is nullhomotopic in $K$. 
 We want to show that $f$ is nullhomotopic in $M$. 
 Equivalently, we suppose that $f$ is null-homotopic in $N$ (which is homotopy equivalent to $K$ by an extension of the map $p:M\to K$ to $N$), and we want to show that it is still null-homotopic in $N \setminus K$ (which is homotopy equivalent to $M$).  This can be achieved by perturbing the null-homotopy disk while fixing its boundary, so that this disk becomes transverse to $K$.  Once the disk is transverse to $K$, it is disjoint from $K$, because the codimension of $K$ in $N$ is greater than $2$.
    
\end{proof}

The next lemma says that when there is a map $p:M\rightarrow K$ as in Lemma~\ref{l:complex to mfld}, one can construct a map $K^{(2)}\rightarrow M$ that does not collapse big set into small set. This will be useful for us to conclude that $M$ has large 1-width when $K$ has large 1-width.

\begin{lemma}\label{l:inverse map}
    Let $f:X\rightarrow Y$ be a  surjective map between two compact Riemannian polyhedra such that 
    \begin{enumerate}
        \item every open cover of $Y$ has a refinement $\{U_\alpha\}$ such that $f^{-1}(U_\alpha)$ is simply connected for each $\alpha$.    
        \item $d(f(x),f(y))\leq d(x,y)+1$ for all $x,y\in X$.        
\end{enumerate}
    Then there exists a map $g:Y^{(2)}\rightarrow X$ such that $d(g(x),g(y))\geq d(x,y)-3$ for all $x,y\in Y^{(2)}$.
\end{lemma}
\begin{proof}
By hypothesis, we can choose an open cover $U_\alpha$ of $Y$ such that $\diam(U_\alpha)\leq 1$ and $f^{-1}(U_\alpha)$ is simply connected for each $\alpha$.
  Let $2k$ be the Lebesgue number of the cover.
By hypothesis, we can choose another cover $\{V_\beta\}$ of $Y$ such that $\diam(V_\beta)\leq k$ and $f^{-1}(V_\beta)$ is path connected for each $\beta$.
 Take a fine triangulation of $Y$ such that any simplex in $Y$ is supported in some $V_\beta$.
 Now, first define $g$ on each vertex $x$ such that $f\circ g(x)=x$. Then define $g$ on an edge $[v,w]$ by first choosing a $V_\beta$ such that $[v,w]\subset V_\beta$, and then sending $[v,w]$ to a path connecting $g(v)$ and $g(w)$ in $f^{-1}(V_\beta)$. 
Since $\diam(V_\beta)\leq k$ for all $\beta$ and $f\circ g$ of any edge in $Y$ lives in some $V_\beta$,  we have that for any $2$-simplex $\sigma$, $f\circ g(\partial \sigma)$ is supported in a set of diameter at most $2k$.  
In fact, $\sigma \cup f \circ g(\partial \sigma)$ also has diameter at most $2k$.
Hence there exist a $U_\alpha$ such that $\sigma \cup f\circ g(\partial \sigma)\subset U_\alpha$.
Since $f^{-1}(U_\alpha)$ is simply connected, we can extend the map $g:\partial \sigma\rightarrow f^{-1}(U_\alpha)$ to the whole $2$-simplex $\sigma$.
In this way, we get a map $g:Y^{(2)}\rightarrow X$ with the property that, $f\circ g(x)=x$ for each vertex in $Y^{(2)}$ and $f\circ g(\sigma)$ is supported in some $U_\alpha$ for each simplex $\sigma$ in $Y^{(2)}$.  
Since  $U_\alpha$ contains $\sigma$ as well as $f \circ g(\partial \sigma)$ and $\diam(U_\alpha)\leq 1$, it follows that $d(f\circ g(x),x)\leq 1$ for any $x\in Y^{(2)}$.  
Therefore for any $x,y\in Y^{(2)}$, we have the following
\begin{align*}
    d(g(x),g(y))& \geq d(f\circ g(x),f\circ g(y))-1\\
    & \geq d(x,y)-d(x,f\circ g(x))-d(f\circ g(y),y)-1\\
    &\geq d(x,y)-3
    \end{align*}
where the first inequality follows because $f$ changes the distance by at most 1 and the second inequality is triangle inequality.
\end{proof}
 Now we are ready to prove Proposition~\ref{p:mfld}.
\begin{proof}[Proof of Propostion~\ref{p:mfld}]
After rescaling the metric of $X_n$, we can assume that $\{\UW_1(X_n)\}$ is unbounded and $\{\UW_1(\widetilde{X_n})\}$ is bounded. Furthermore, after subdividing, we can assume that the diameter of each simplex in $X_n$ is at most 1.
Pick $r_n>0$ such that the covering map $\widetilde{X_n}\rightarrow X_n$ restricts to isometry on any set that has diameter at most  $2r_n$.
In particular, any loop in $X_n$ of diameter at most $2 r_n$ is nullhomotopic.
Such $r_n$ exists because $X_n$ is  compact and without loss of generality we can assume that  $r_n\leq 1$.
Applying  Lemma~\ref{l:complex to mfld}, we pick a closed manifold $Z_n$ and a $\pi_1$-injective map $p_n:Z_n\rightarrow X_n$ such that $|d(p_n(x),p_n(y))- d(x,y)|\leq \frac{r_n}{4}$ for any $x,y\in Z_n$. 

 First we prove that $\{\UW_1(Z_n)\}$ is unbounded. Let $L_n:=\UW_1(Z_n)$ and  $\gamma_n:Z_n\rightarrow \Gamma_n$ be a map to the graph $\Gamma_n$ such that each fiber has diameter at most $L_n+1$.
By Lemma~\ref{l:inverse map}, there exists a map $g_n:X_n^{(2)}\rightarrow Z_n$ such that for any $A\subset Z_n$, we have $\diam(g_n^{-1}(A))\leq \diam(A)+3$. 
It follows that, $\UW_1(X_n^{(2)})\leq L_n+1+3$.
By Lemma~\ref{l:width of subcomplex} we obtain that $\UW_1(X_n)\leq L_n+4+2\dim(X_n)$.
Since $\{\dim(X_n)\}$ is bounded and $\{\UW_1(X_n)\}$ is unbounded, we have $\{L_n\}$ is unbounded.

    Next we prove that the sequence $\{\UW_1(\widetilde{Z_n})\}$ is bounded.
    Let $\widetilde{p_n}:\widetilde{Z_n}\to \widetilde{X_n}$ denote the lift of $p_n:Z_n\to X_n$ to the universal covers.
    We first show that it is enough to prove that  $d(a,b)\leq \frac{r_n}{2}$ if $d(\widetilde{p_n}(a),\widetilde{p_n}(b))\leq \frac{r_n}{4}$ for any $a,b\in \widetilde{Z_n}$.
    Assuming this is true, pick any two points $a,b\in \widetilde{Z_n}$.
    Let $\gamma$ be a length-minimizing geodesic between $\widetilde{p_n}(a)$ and $\widetilde{p_n}(b)$. 
    Divide $\gamma$ into $k_n:=\lceil\frac{4\cdot d(\widetilde{p_n}(a),\widetilde{p_n}(b))}{r_n}\rceil$ many consecutive sub-geodesics $[c_i,c_{i+1}]$ each having length at most  $\frac{r_n}{4}$ where $c_0=\widetilde{p_n}(a)$ and $c_{k_n}=\widetilde{p_n}(b)$.
    Let $c'_i\in \widetilde{p_n}^{-1}(c_i)$.
Since $d(c_i,c_{i+1})\leq \frac{r_n}{4}$, by our assumption $d(c'_i,c'_{i+1})\leq \frac{r_n}{2}$.
    We obtain
    \begin{align*}
        d(a,b)\leq \sum_{i=0}^{k_n-1} d(c'_i,c'_{i+1})&\leq k_n\cdot \frac{r_n}{2}\\
        &\leq \left(\frac{4 \cdot d(\widetilde{p_n}(a),\widetilde{p_n}(b))}{r_n}+1\right)\cdot \frac{r_n}{2}\\
        &\leq 2d(\widetilde{p_n}(a),\widetilde{p_n}(b))+1.        
    \end{align*}    
It follows that $\UW_1(\widetilde{Z_n})\leq 2\cdot \UW_1(\widetilde{X_n})+1$ and therefore $\{\UW_1(\widetilde{Z_n})\}$ is bounded since $\{\UW_1(\widetilde{X_n})\}$ is bounded.  
  Therefore,  it remains to prove that  $d(a,b)\leq \frac{r_n}{2}$ if $d(\widetilde{p_n}(a),\widetilde{p_n}(b))\leq \frac{r_n}{4}$ for any $a,b\in \widetilde{Z_n}$.

Let $q_n:\widetilde{Z_n}\rightarrow Z_n$ and $q_n':\widetilde{X_n}\rightarrow X_n$ be the covering maps. The following commutative diagram will be useful to follow the proof.
\[
\begin{tikzcd}
    \widetilde{Z_n}\arrow{d}{q_n}\arrow{r}{\widetilde{p_n}} & \widetilde{X_n}\arrow{d}{q'_n}\\
    Z_n\arrow{r}{p_n} & X_n
\end{tikzcd}
\]

First we claim that $q_n$ restricted to any set of diameter $\leq \frac{r_n}{2}$ is  an isometry, i.e., $d(q_n(a),q_n(b))=d(a,b)$ if $d(a,b)\leq \frac{r_n}{2}$. 
   If not, then a length-minimizing geodesic between $q_n(a)$ and $q_n(b)$ in $Z_n$ and the $q_n$-image of a length-minimizing geodesic between $a$ and $b$ form a homotopically nontrivial loop $c$ of diameter at most $r_n$.
    Since $p_n$ is $\pi_1$-injective and changes distance by at most $\frac{r_n}{4}$, $p_n(c)$ is a homotopically nontrivial loop of diameter at most $2 r_n$.
    This is a contradiction to the assumption that any loop of diameter at most $2r_n$ is nullhomotopic in $X_n$.
   Therefore $d(q_n(a),q_n(b))=d(a,b)$ if $d(a,b)\leq \frac{r_n}{2}$.

  Finally we are ready to show that $d(a,b)\leq \frac{r_n}{2}$ if $d(\widetilde{p_n}(a),\widetilde{p_n}(b))\leq \frac{r_n}{4}$ for any $a,b\in \widetilde{Z_n}$.  We take $a,b\in \widetilde{Z_n}$ so that $d(\widetilde{p_n}(a),\widetilde{p_n}(b))\leq \frac{r_n}{4}$. 
   Take a length-minimizing geodesic $\gamma$ between $\widetilde{p_n}(a)$ and $\widetilde{p_n}(b)$.
    Note that $\widetilde{p_n}^{-1}(\gamma)$ is contained in $q_n^{-1}p_n^{-1}(q_n'(\gamma))$ by the commutativity of the above diagram.
    Since the diameter of $\gamma$ is at most $\frac{r_n}{4}$, $q'_n(\gamma)$ is isometric to $\gamma$.
    Since $q'_n(\gamma)$ is connected, applying Lemma~\ref{l:complex to mfld}(1) to $p_n$, we have that $p_n^{-1}(q_n'(\gamma))$ is a connected set.
   Since the diameter of $q_n'(\gamma)$ is at most $\frac{r_n}{4}$ and $p_n$ changes distance by at most $\frac{r_n}{4}$, we have that $p_n^{-1}(q_n'(\gamma))$ is a connected set of diameter at most $\frac{r_n}{2}$.
   Since $q_n$ restricted to a set of diameter $\leq \frac{r_n}{2}$ is an isometry by our previous claim and $q_n$ is a covering map, it follows that $q_n^{-1}p_n^{-1}(q_n'(\gamma))$ is disjoint union of sets each of which are isometric to $p_n^{-1}(q_n'(\gamma))$.
   Only one of these components is contained in  $\widetilde{p_n}^{-1}(\gamma)$ because $p_n$ is $\pi_1$-injective.
   Therefore the diameter of $\widetilde{p_n}^{-1}(\gamma)$ is at most $\frac{r_n}{2}$ and in particular $d(a,b)\leq \frac{r_n}{2}$.
    This completes the proof that $\{\UW_1(\widetilde{Z_n)}\}$ is bounded.
\end{proof}

Next our goal is to improve the output $\{Z_n\}$ of the  Proposition~\ref{p:mfld} so that each $Z_n$ has dimension four given that the sequence $\{\dim(X_n)\}$ is bounded.
The key to achieve this is the following.

\begin{lemma}\label{l:width of subcomplex}
Let $X$ be a Riemannian manifold and  $Y$ be the $2$-skeleton of some cubulation of $X$ with the extrinsic metric. Then the following holds.
    \begin{enumerate}
        \item $\UW_1(\widetilde{X})\geq \UW_1(\widetilde{Y})$.
        \item If the cubes of the cubulation of $X$ are of diameter at most $k$, then $\UW_1(X)\leq \UW_1(Y)+2k\cdot \dim(X)$.
    \end{enumerate}
    In particular, if $\{X_n\}$ is a sequence of Riemannian polyhedra such that  the diameter of each simplex in $X_n$ is uniformly bounded, $\{\dim(X_n)\}$ is bounded, and the sequence $\displaystyle\frac{\UW_1(X_n)}{\UW_1(\widetilde{X_n})}$ is unbounded.
    Then $\displaystyle\frac{\UW_1(X^{(2)}_n)}{\UW_1(\widetilde{X^{(2)}_n})}$ is unbounded where $X^{(2)}_n$ is equipped with the extrinsic metric.    
\end{lemma}
\begin{proof}
    To prove the first claim, we observe that the isometry $Y\hookrightarrow X$ induces isomorphism between the corresponding fundamental groups. 
Hence $Y\hookrightarrow X$ lifts to an isometric embedding $\widetilde{Y}\rightarrow \widetilde{X}$. It follows that $\UW_1(\widetilde{X})\geq \UW_1(\widetilde{Y})$.

To prove the second claim, choose $d> \UW_1(Y)$ arbitrarily.
    Then there exists a graph $\Gamma$ and a map $f:Y\rightarrow \Gamma$ such that the diameters of fibers of $f$ are at most $d$.
    Now, we will  extend $f$  to $X$ inductively on each skeleton of $X$.
    Suppose, we already have extended  $f$ to the $n$-skeleton of $X$ where $n\geq 2$.
    Abusing notation, we will denote the extension by $f$ as well.
     Since $\Gamma$ is a graph, for   any $(n+1)$-cube $\sigma$ in $X$, $\pi_n(f(\partial \sigma))$ is trivial. Therefore, we can extend $f$ to the $(n+1)$-skeleton such that $f(\sigma)\subset f(\partial\sigma)$ for each $(n+1)$-cube.
    Since cubes in $X$ have diameters at most $k$, this extension increases the diameters of fibers by at most $2k$ amount. 
     This way, in each inductive step the diameters of the fibers get increased by at most $2k$ amount, and hence the fibers of the final map are of diameter at most $d+2k \cdot \dim(X)$.
     Therefore, $\UW_1(X)\leq d+2k \cdot \dim(X)$.
     Since $d> \UW_1(Y)$ was chosen arbitrarily, we obtain the desired claim.
\end{proof}

Note that the $X^{(2)}_n$ in Lemma~\ref{l:width of subcomplex} may not be a Riemannian polyhedron because it is equipped with the extrinsic metric.
In the next proposition, we improve Lemma~\ref{l:width of subcomplex} by producing a sequence of Riemannian polyhedra with the same property.
This is accomplished by replacing the extrinsic metric of $X^{(2)}_n$ by its intrinsic metric.
However, we must ensure that the extrinsic metric and the intrinsic metric are close to one another so that the  ratio $\displaystyle\frac{\UW_1(X^{(2)}_n)}{\UW_1(\widetilde{X^{(2)}_n})}$ remains unbounded with the intrinsic metric. 
For this we will use the Federer--Fleming deformation technique,  
which allows us to project a geodesic in $\widetilde{X}_{n}$ onto its lower skeleta
with only controlled distortion.  
More precisely, we require the following lemma from~\cite[Lemma 5]{G}.

\begin{lemma}[Pushing a cycle to the boundary of a cube]\label{l:pushing cycle}
Suppose that $z$ is a relative $n$-cycle inside a $k$-dimensional cube $Q$ with equal side lengths, with $k > n$. Then there exists an $n$-chain $\tilde{z} \subset \partial Q$ such that  $\partial \tilde{z} = \partial z$ and
$\operatorname{vol}(\tilde{z}) \leq C_k \, \operatorname{vol}(z)$.
\end{lemma}

\begin{proposition}[$2$-dimensional reduction]\label{p:2dim}
Suppose that $\{X_n\}_{n=1}^\infty$ is a sequence of compact Riemannian polyhedra such that $\{\dim(X_n)\}$ is bounded and the ratio $\displaystyle\frac{\UW_1(X_n)}{\UW_1(\widetilde{X_n})}$ is unbounded.
Then there is a related sequence $\{Y_n\}_{n=1}^\infty$ of compact 2-dimensional Riemannian polyhedra such that the ratio $\displaystyle\frac{\UW_1(Y_n)}{\UW_1(\widetilde{Y_n})}$ is unbounded.
\end{proposition}

\begin{proof}
    Let $\{X_n\}_{n=1}^\infty$ be a sequence of compact Riemannian polyhedra as in the hypothesis.
    By Proposition~\ref{p:mfld}, we can assume that each $X_n$ is a closed Riemannian manifold satisfying the same hypotheses.
    Furthermore, after scaling we can assume that $\{\UW_1(X_n)\}$ is unbounded and $\{\UW_1(\widetilde{X_n})\}$ is bounded.

By \cite[Theorem 1.1]{Bowditch}, any smooth closed Riemannian manifold admits a smooth cubulation such that if we equip each cube with the metric of a standard Euclidean cube (up to some scaling), then the resulting path metric on the manifold induced from the cubulation differs from the original metric by at most some multiplicative constant $L$ that depends only on the dimension of the manifold (also see \cite{BoDG}).
Applying this to each $X_n$, we can assume that each $X_n$ is equipped with a cubulation and a path metric where each cube is a standard Euclidean cube (up to some scaling).
Since $\{\dim(X_n)\}$ is bounded, this path metric differs from the original metric by a multiplicative constant $L$ for all $X_n$.
Thus for the rest of the proof we can assume that $X_n$ is equipped with the path metric coming from the cubulation.

Let $Y^e_n$ and $Y_n$ denote the 2-skeleton $X_n^{(2)}$ of $X_n$ equipped with the extrinsic and intrinsic metric respectively. By Lemma~\ref{l:width of subcomplex}, $\UW_1(\widetilde{X_n})\geq \UW_1(\widetilde{Y^e_n})$ and $\UW_1(X_n)\leq \UW_1(Y^e_n)+1$ assuming the cubulation of $X_n$ is fine enough.
Since $\{\UW_1(X_n)\}$ is unbounded and $\{\UW_1(\widetilde{X_n})\}$ is bounded, it follows that $\{\UW_1(Y^e_n)\}$ is unbounded and $\{\UW_1(\widetilde{Y^e_n})\}$ is bounded.
However, $Y^e_n$ is not a Riemannian polyhedron, whereas $Y_n$ is.
Since $\{\UW_1(Y^e_n)\}$ is unbounded and the identity map $Y_n\rightarrow Y^e_n$ is distance non-increasing, $\{\UW_1(Y_n)\}$ is unbounded.
Therefore, it remains to show that $\{\UW_1(\widetilde{Y_n})\}$ is bounded.
We already know that $\{\UW_1(\widetilde{Y^e_n})\}$ is bounded.
 It is enough to show that we can choose a cubulation of $X_n$ so that the difference between $\widetilde{Y^e_n}$ and $\widetilde{Y_n}$ is uniformly close to each other for all $n$.

Suppose  $a,b\in \widetilde{Y_n}$. We claim that there is a $C>0$ such that 
\[d_{\widetilde{Y_n}}(a,b)\leq C\cdot d_{\widetilde{Y^e_n}}(a,b) \qquad \forall n.
\]

To prove the claim, we take a length minimizing geodesic $\gamma$ between $a$ and $b$ in $\widetilde{X_n}$ whose length is $d_{\widetilde{Y^e_n}}(a,b)$.
Then we  apply Lemma~\ref{l:pushing cycle} repeatedly to push $\gamma$  consecutively onto lower and lower skeleta of each cube in $\widetilde{X_n}$, acquiring some constant factor to its length each time, to get a path $\gamma'$ between $a$ and $b$ in the two-skeleton $\widetilde{Y_n}$.
By construction, $\length(\gamma')\leq C_n\cdot d_{\widetilde{Y^e_n}}(a,b)$ where $C_n$ depends only on the $\dim \widetilde{X_n}$.
Since the sequence $\{\dim{\widetilde{X_n}}\}$ is bounded, we have $C:=\max_n{C_n}$ is finite.
Since $d_{\widetilde{Y_n}}(a,b)\leq \length(\gamma')$, we obtain $d_{\widetilde{Y_n}}(a,b)\leq C\cdot d_{\widetilde{Y^e_n}}(a,b)$ for all $n$.

Since $a,b\in \widetilde{Y_n}$ were chosen arbitrarily, we conclude
    \begin{align*}
        \UW_1(\widetilde{Y_n})\leq C \cdot \UW_1(\widetilde{Y^e_n}) \qquad \forall n.
\end{align*}

Since the sequence $\{\UW_1(\widetilde{Y^e_n})\}$ is bounded, $\{\UW_1(\widetilde{Y_n})\}$ is bounded.

\end{proof}
Proposition~\ref{p:mfld} and Proposition~\ref{p:2dim} together yield the following.
\begin{theorem}[$4$-Manifold reduction]\label{t:4dim}
Suppose that $\{X_n\}_{n=1}^\infty$ is a sequence of compact Riemannian polyhedra such that the ratio $\displaystyle\frac{\UW_1(X_n)}{\UW_1(\widetilde{X_n})}$ is unbounded and $\dim(X_n)$ is bounded. Then we can construct a related sequence $\{Z_n\}_{n=1}^\infty$ of closed Riemannian 4-manifolds such that the ratio $\displaystyle\frac{\UW_1(Z_n)}{\UW_1(\widetilde{Z_n})}$ is unbounded.
\end{theorem}
\begin{proof}
    We first apply Proposition~\ref{p:2dim} to obtain a sequence of 2-dimensional Riemannian polyhedra $\{Y_n\}$ such that $\frac{\UW_1(Y_n)}{\UW_1(\widetilde{Y_n})}$ is unbounded. Then we can apply Proposition~\ref{p:mfld} to obtain a sequence of closed Riemannian manifolds $\{Z_n\}$ such that $\dim(Z_n)=2\dim(Y_n)=4$ and the ratio is $\frac{\UW_1(Z_n)}{\UW_1(\widetilde{Z_n})}$ is unbounded.
\end{proof}

\bibliography{uwbib}
\bibliographystyle{amsalpha}
\end{document}